\documentclass[11pt,letterpaper]{amsart}
\usepackage[utf8]{inputenc}
\usepackage{float}
\usepackage{amsmath}
\usepackage{amsfonts}
\usepackage{amssymb}
\usepackage{amsthm}
\usepackage{graphicx}
\usepackage{lmodern}
\usepackage{tikz}
\usepackage{caption}
\usepackage{subcaption}
\usepackage{color}
\usetikzlibrary{positioning}
\usepackage{siunitx}
\usepackage{pgfplots}
\usepackage{ mathrsfs }
\usepackage{hyperref}
\pgfplotsset{compat=newest}
\usepgfplotslibrary{fillbetween}

\hypersetup{
  colorlinks   = true, 
  urlcolor     = blue, 
  linkcolor    = blue, 
  citecolor   = red 
}

\graphicspath{ {figures/} }

\newtheorem{theorem}{Theorem}[section]

\newtheorem{cor}[theorem]{Corollary}

\newtheorem{lemma}[theorem]{Lemma}
\newtheorem{proposition}[theorem]{Proposition}

\newtheorem{conjecture}[theorem]{Conjecture}

\newtheorem{definition}[theorem]{Definition}

\newcommand{\C}{\mathbb{C}}
\newcommand{\R}{\mathbb{R}}

\newcommand{\Q}{\mathbb{Q}}
\newcommand{\Z}{\mathbb{Z}}
\newcommand{\N}{\mathbb{N}}

\title{Circle Packings from Tilings of the Plane}
\author{Philip Rehwinkel, Ian Whitehead, David Yang, Mengyuan Yang}
\date{}

\begin{document}

\begin{abstract} We introduce a new class of fractal circle packings in the plane, generalizing the polyhedral packings defined by Kontorovich and Nakamura. The existence and uniqueness of these packings are guaranteed by infinite versions of the Koebe-Andreev-Thurston theorem. We prove structure theorems giving a complete description of the symmetry groups for these packings. And we give several examples to illustrate their number-theoretic and group-theoretic significance.
\end{abstract}

\maketitle

\section{Introduction}

The well-known Apollonian circle packing can be constructed from a set of four base circles, and four dual circles, as shown in Figure \ref{fig:apollonian}. The orbit of the base circles under the group generated by reflections through the dual circles is the packing, an infinite fractal set of circles. Beyond their aesthetic appeal, Apollonian packings have properties of great interest in number theory, group theory, and fractal geometry. 

\begin{figure}[h]
\begin{subfigure}{.49\textwidth}
\begin{tikzpicture}
\clip (0,0) circle(3.0);
\draw[thick, blue] (1.000, 0) circle(0.8660);
\draw[thick, blue] (-0.500, 0.866) circle(0.8660);
\draw[thick, blue] (-0.500, -0.866) circle(0.8660);
\draw[thick, blue] (0, 0) circle(1.866);
\draw[thick, red] (0, 0) circle(0.5000);
\draw[thick, red] (1.866, 3.232) circle(3.232);
\draw[thick, red] (1.866, -3.232) circle(3.232);
\draw[thick, red] (-3.732, 0) circle(3.232);
\end{tikzpicture}
\caption{Base (blue) and dual (red) circles}
\end{subfigure}
\begin{subfigure}{.49\textwidth}
\includegraphics[width=\textwidth]{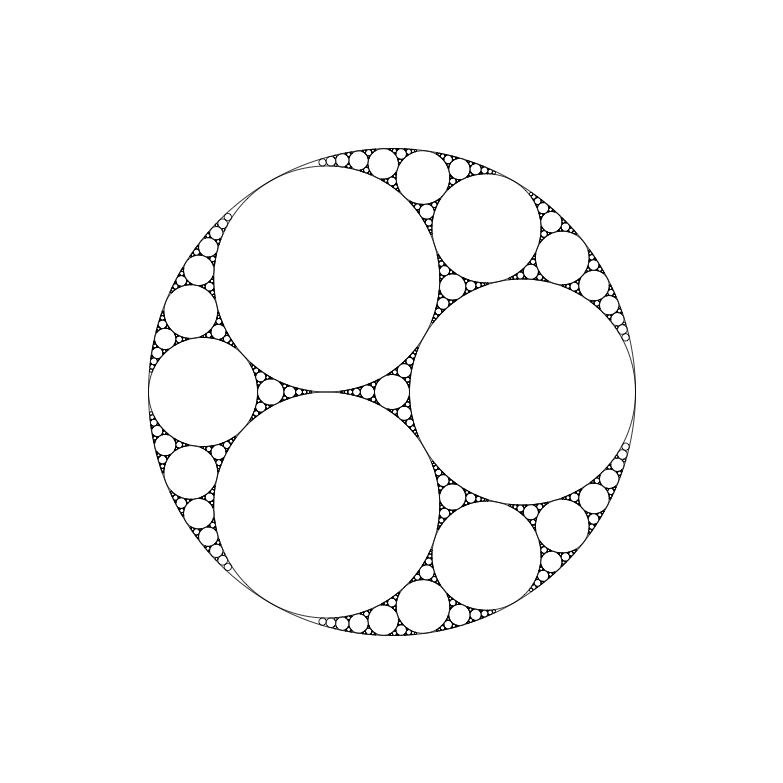}
\caption{Apollonian packing}
\end{subfigure}
\caption{Constructing the Apollonian packing}
\label{fig:apollonian}
\end{figure}
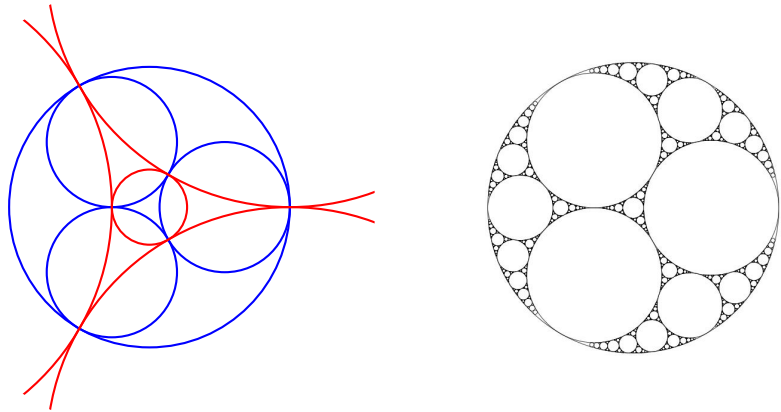

In \cite{KontorovichNakamura}, Kontorovich and Nakamura define polyhedral circle packings, generalizing the Apollonian packing construction. Any circle configuration has a tangency graph, with a vertex for each circle and an edge for each tangency between circles. In the case of the Apollonian packing, both the base circles and the dual circles have tetrahedral tangency graphs. In general, one can start with a finite set of base circles whose tangency graph is the graph of any polyhedron and a finite set of dual circles whose tangency graph is the graph of the dual polyhedron. The orbit of the base circles under the group generated by the dual circles is a polyhedral packing. Polyhedral packings encompass many of the generalizations of the Apollonian packing that have been studied previously. For example, the packing introduced by Guettler and Mallows \cite{GuettlerMallows} is the octahedral packing; the $\Q[\sqrt{-2}]$ packing studied by Stange \cite{Stange} is the cubic packing. 

In this article, we study packings which originate from infinite configurations of base and dual circles, a further generalization. A particularly symmetric example is shown in Figure \ref{fig:square}. In this example, both the base and dual circles have the square lattice as their tangency graphs. Again, the orbit of the base circles under the group generated by the dual circles is a fractal set of circles. We call this object the square packing. 

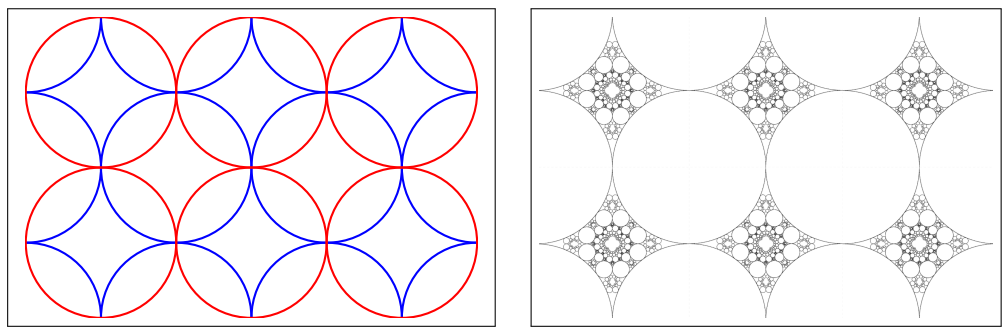
\begin{figure}[H]
\begin{subfigure}{.45\textwidth}
\fbox{
\begin{tikzpicture}
\clip (-2,-2) -- (-2, 2) -- (4,2) -- (4,-2);
\draw[thick, blue] (-2,-2) circle(1.0);
\draw[thick, blue] (0,-2) circle(1.0);
\draw[thick, blue] (2,-2) circle(1.0);
\draw[thick, blue] (4,-2) circle(1.0);
\draw[thick, blue] (-2,0) circle(1.0);
\draw[thick, blue] (0,0) circle(1.0);
\draw[thick, blue] (2,0) circle(1.0);
\draw[thick, blue] (4,0) circle(1.0);
\draw[thick, blue] (-2,2) circle(1.0);
\draw[thick, blue] (0,2) circle(1.0);
\draw[thick, blue] (2,2) circle(1.0);
\draw[thick, blue] (4,2) circle(1.0);

\draw[thick, red] (-1,-1) circle(1.0);
\draw[thick, red] (1,-1) circle(1.0);
\draw[thick, red] (3,-1) circle(1.0);
\draw[thick, red] (-1,1) circle(1.0);
\draw[thick, red] (1,1) circle(1.0);
\draw[thick, red] (3,1) circle(1.0);
\end{tikzpicture}
}
\caption{Base (blue) and dual (red) circles}
\end{subfigure}
\hspace{.08\textwidth}
\begin{subfigure}{.45\textwidth}
\fbox{\includegraphics[width=1.058\textwidth]{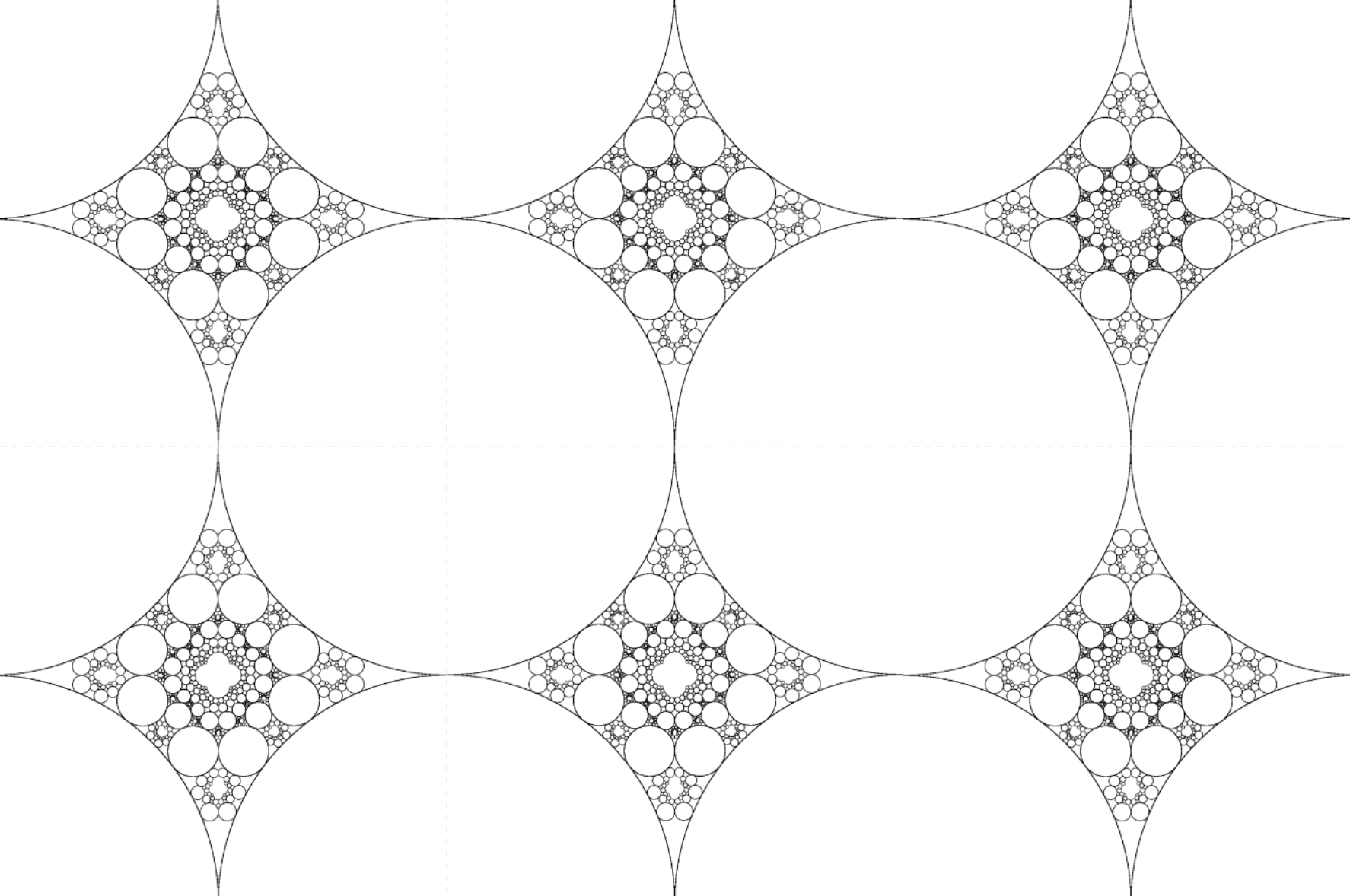}}
\caption{Square packing}
\end{subfigure}
\caption{Constructing the square packing}
\label{fig:square}
\end{figure}

In general, we work with a base circle configuration $B$ and a dual configuration $\hat{B}$ whose tangency graphs determine dual tilings, or  cellular decompositions of the sphere $\hat{\C}$ or the plane $\C$--see Definition \ref{basedual}. Tilings of the sphere give rise to polyhedral packings, while tilings of the plane give rise to new examples. In both cases, the final packings have similar geometric properties. For example, the circles in the packing are pairwise disjoint or tangent, they can be oriented with disjoint interiors, and the interiors are dense in the ambient space. The symmetry groups of polyhedral packings and of our new examples have similar structure, as we prove in Section 4. And some of our new examples have integrality properties, raising number-theoretic questions. 

One motivation to introduce these new packings comes from the literature on the Koebe-Andreev-Thurston theorem and its generalizations. The Koebe-Andreev-Thurston theorem \cite{Koebe, Andreev, Thurston} is the remarkable result that any pair of finite graphs $G$, $\hat{G}$ representing a polyhedron and its dual can be realized as the tangency graphs for a pair of dual circle configurations $B$, $\hat{B}$. Moreover, the circle configurations $B$, $\hat{B}$ are unique up to conformal automorphism of $\hat{\C}$ (see Theorem \ref{KAT1} for a precise statement). The Koebe-Andreev-Thurston theorem implies the existence and uniqueness up to M\"{o}bius transformation of a circle packing for every polyhedron. 

It is natural to try to extend this theorem to infinite graphs. Important work of Beardon-Stephenson and of Schramm achieves this in many cases \cite{BeardonStephenson, Schramm:ExistenceandUniqueness, Schramm:Rigidity}, following a constructive approach suggested by Thurston. Stephenson's text \cite{Stephenson} synthesizes this work. One result is that any infinite graph $G$ representing a triangulation of the plane can be realized as the tangency graph for a circle configuration $B$, and $B$ is unique up to conformal automorphism of $\C$ (see Theorem \ref{KAT2}). The proofs use deep geometric ideas---mappings between circle configurations give a discrete analogue of the Riemann mapping theorem. 

Just as the construction of polyhedral circle packings relies on the finite Koebe-Andreev-Thurston theorem, our construction relies on its infinite generalizations. In Section 3, we state various versions of the theorem which imply the existence and uniqueness up to conformal automorphism of many of our packings. We also make a more general conjecture which would imply existence and uniqueness in all cases.  Our work combines geometric ideas from the infinite Koebe-Andreev-Thurston theorem with arithmetic ideas from the Apollonian packing and its relatives. 

Another motivation comes from the definitions of crystallographic packings in \cite{KontorovichNakamura}. These are a class of packings which generalize the Apollonian packing, and encompass many known examples of circle and sphere packings. To construct crystallographic circle packings, one can start with a geometrically finite reflection group which acts on  $\mathbb{H}^3$ with finite covolume. Each wall of the fundamental chamber intersects the spherical boundary of $\mathbb{H}^3$ in a circle. Suppose that these circles are partitioned into two sets, a ``cluster'' and ``cocluster,'' such that circles in the cluster are pairwise disjoint or tangent, and each circle in the cluster is disjoint, tangent, or orthogonal to each circle in the cocluster. Then the orbit of the cluster under the group generated by reflections across the cocluster is a crystallographic packing. All polyhedral packings are crystallographic, but not all crystallographic packings are polyhedral. 

Kontorovich and Nakamura classify superintegral crystallographic circle packings. Every superintegral packing arises from an arithmetic finite-covolume reflection group acting on $\mathbb{H}^3$. There are finitely many such groups up to commensurability, tabulated in \cite{ScharlauWalhorn}, and all the non-cocompact groups in the tabulation give rise to packings. The most interesting case is the Bianchi group of the Eisenstein integers. This is represented by the Coxeter-Dynkin diagram:

\begin{center}

\begin{tikzpicture}[scale=1]
\draw[thin] (1.1,0) -- (1.9,0);
\draw[thin] (2.1,0) -- (2.9,0);
\draw[thin] (3.1,0.02) -- (3.9,0.02);
\draw[thin] (3.1,-0.02) -- (3.9,-0.02);
\draw[thin] (3.1,0.06) -- (3.9,0.06);
\draw[thin] (3.1,-0.06) -- (3.9,-0.06);
\filldraw [black] (1,0) circle (3pt);
\filldraw [black] (2,0) circle (3pt);
\filldraw [black] (3,0) circle (3pt);
\filldraw [black] (4,0) circle (3pt);
\end{tikzpicture}

\end{center}

\noindent Each vertex represents a circle, and the edge types indicate angles between circles. There is a finite-index subgroup with the diagram:

\begin{center}

\begin{tikzpicture}[scale=1]
\draw[line width=2pt] (1.1,0) -- (1.9,0);
\draw[thin] (2.1,0) -- (2.9,0);
\draw[thin] (3.1,0.02) -- (3.9,0.02);
\draw[thin] (3.1,-0.02) -- (3.9,-0.02);
\draw[thin] (3.1,0.06) -- (3.9,0.06);
\draw[thin] (3.1,-0.06) -- (3.9,-0.06);
\draw[line width=2pt] (4.1,0) -- (4.9,0);
\filldraw [black] (1,0) circle (3pt);
\filldraw [black] (2,0) circle (3pt);
\filldraw [black] (3,0) circle (3pt);
\filldraw [black] (4,0) circle (3pt);
\filldraw [black] (5,0) circle (3pt);
\end{tikzpicture}

\end{center}

\noindent See Figure \ref{fig:triangular3} (A), where this diagram is realized as a set of five circles. If the left vertex is the cluster and the remaining vertices are the cocluster, the result is the triangular packing of Figure \ref{fig:triangular4}. Similarly, if the right vertex is the cluster and the remaining vertices are the cocluster, the result is the hexagonal packing of Figure \ref{fig:triangular5}. These are fundamental examples of superintegral crystallographic packings. We conjecture that they are not polyhedral, or commensurate (on the level of hyperbolic reflection groups) to any polyhedral packing. Nevertheless, our construction allows us to work with these packings similarly to polyhedral packings. One might hope to realize all superintegral crystallographic circle packings with our construction.


A final motivation comes from limits in polyhedral packing families. In \cite{Ahmed2}, fractal dimensions for many polyhedral packings are computed. It is observed that for some sequences of polyhedra--pyramids, prisms, antiprisms, etc.--the fractal dimensions converge to a limit. More surprisingly, the packings themselves converge to a well-defined limit. This phenomenon is illustrated in Figure \ref{packinglimit}. A configuration of circles in the polyhedral packing for the 100-sided prism is shown. As the number of sides increases, this configuration approaches a configuration found in the square packing of Figure \ref{fig:square}. The limit of the prism packings is our square packing. 

\begin{figure}[H]
\begin{subfigure}{.49\textwidth}
\includegraphics[width=\textwidth]{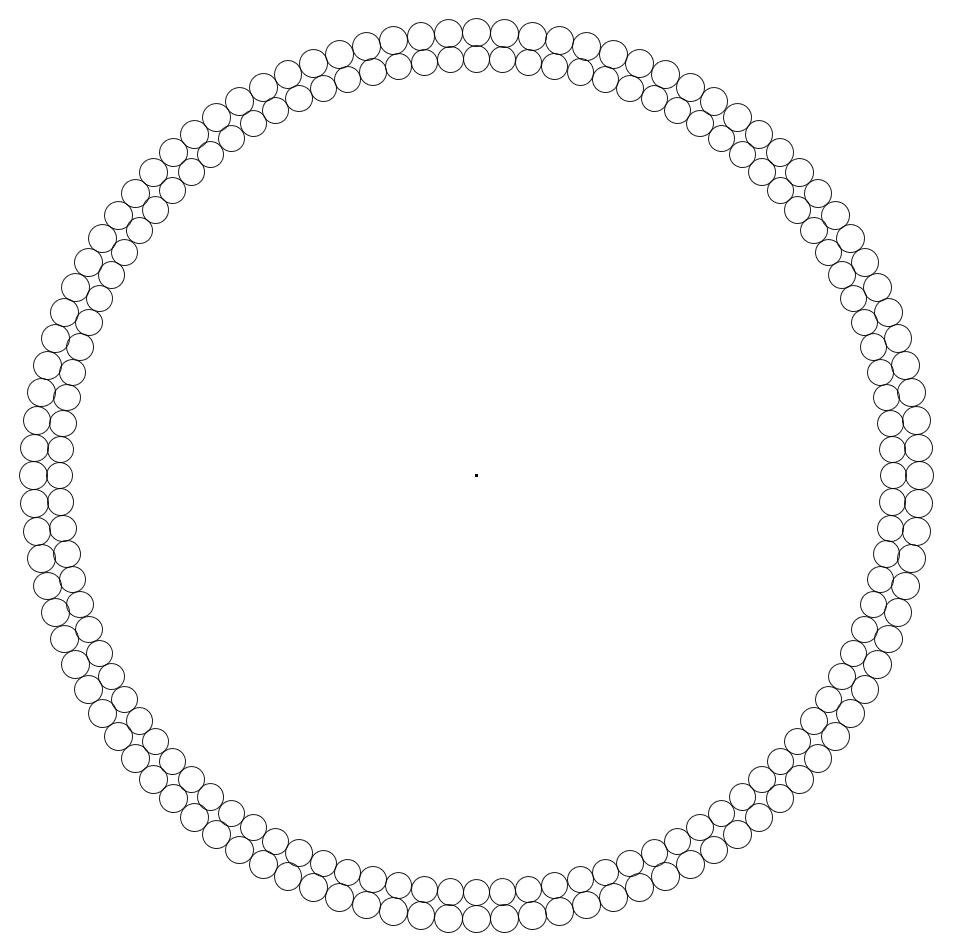}
\caption{100-sided prism configuration}
\end{subfigure}
\begin{subfigure}{.49\textwidth}
\includegraphics[width=\textwidth]{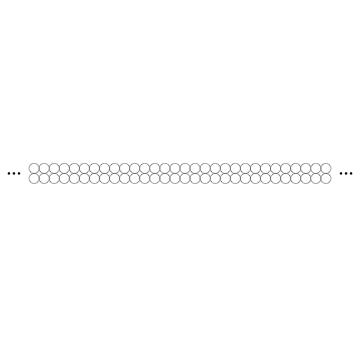}
\caption{Square configuration} 
\end{subfigure}
\caption{Limiting configurations}
\label{packinglimit}
\end{figure}

The article \cite{Ahmed2} focuses on examples, and does not give a general criterion for a sequence of circle packings to converge. Some of the limit packings are also polyhedral--e.g. the limit of the pyramid packings is the original Apollonian packing--and others are not polyhedral but satisfy our definition. We discuss all these examples in Section 5. It would be very interesting to know whether our definition gives the closure of the space of polyhedral packings. This would require a more systematic development of the notion of a limit in this space.

This article is structured as follows. In Section 2, Definitions \ref{basedual} and \ref{packing} describe the class of packings we study. We prove some of the fundamental geometric consequences of these definitions, and compare them to the definitions of polyhedral, crystallographic, and Kleinian packings. In Section 3, we recall versions of the Koebe-Andreev-Thurston theorem and propose our extension, Conjecture \ref{KATconjecture}. In Section 4, we analyze the symmetries of our packings. Theorems \ref{gamma1}, \ref{gamma2}, and \ref{semidirectproduct} give a complete description of the symmetry group. In Section 5, we give three main examples of our packings, which we call the triangular, square, and hexagonal packings. We focus on their arithmetic properties--integrality, quadratic and linear forms. In Section 6, we give a broader class of examples, with a focus on symmetries. Theorem \ref{wallpaper} shows that all 17 wallpaper groups appear in the symmetry groups of packings.  

In future work, we hope to elaborate on the number theory of the packings described here. As indicated in Section 5, many of our packings have integral curvatures. Is there an asymptotic formula for curvatures in the packing, as in \cite{KontorovichOh}? Is there a local-to-global principle for curvatures, as conjectured in \cite{GLMWYNumTheory}? These questions can be answered by the methods of \cite{FuchsStangeZhang} for some periodic packings. They become much more subtle if the base configuration lacks any symmetry. It is possible to construct integral, aperiodic packings using the same refinement method as in the proof of Theorem \ref{wallpaper}. These packings are geometrically interesting, and their number theory remains to be explored. 

\subsection{Acknowledgements} We thank Daniel Allcock, Arthur Baragar, Michael Dougherty, Cathy Hsu, Anna Felickson, Alex Kontorovich, Alice Mark, and Kate Stange for helpful conversations relating to this project. We thank Nooria Ahmed, William Ball, Ellis Buckminster, Emilie Rivkin, Dylan Torrance, Jake Viscusi, Runze Wang, and Gary Yang for raising some of the questions that influenced our work. We are grateful to Swarthmore College for funding the summer research project that led to this article, and we thank all the faculty and students in the Department of Mathematics and Statistics who helped create a productive research community. 

\section{Definitions and Basic Properties}

The extended complex plane $\hat{\C}$ is $\C \cup \lbrace \infty \rbrace$, with the topology of the sphere. Our packings consist of oriented generalized circles, i.e. circles and lines, in $\hat{\C}$. Each generalized circle divides $\hat{\C}$ into two simply connected regions. Choosing an orientation for the circle is equivalent to choosing one of these regions to be the interior, and the other to be the exterior. When we refer to circles in this article, we always mean oriented generalized circles. Oriented generalized circles have an action by the group of holomorphic and antiholomorphic M\"{o}bius transfomations $\mathrm{M\ddot{o}b} \cong \mathrm{SL}_2(\C) \rtimes \Z/2\Z$. For more details on this setup, see \cite{LagariasMallowsWilks}.

For any collection of circles $B$, we can associate its tangency graph $G_B$ which has a vertex for every circle and an edge between each pair of tangent circles. We will say that a collection of circles accumulates at a point $x\in \hat{\C}$ if any open neighborhood of $x$ contains infinitely many circles from this collection. We now define the notions of base and dual circle configurations, which are the starting point for the packings we construct. 

\begin{definition}\label{basedual} Let $B$ and $\hat{B}$ be two collections of oriented generalized circles, with tangency graphs $G_B$ and $G_{\hat{B}}$, respectively. Then $B$ is called a base configuration and $\hat{B}$ is called a dual configuration if the following properties hold:

\begin{enumerate}
    \item The circles in $B$ are pairwise disjoint or tangent, with disjoint interiors, and the same holds for the circles in $\hat{B}$.
    \item The tangency graphs $G_{B}$ and $G_{\hat{B}}$ are each nontrivial, connected, and are duals of each other.
    \item If a circle in $B$ and a circle $\hat{B}$ intersect, they do so orthogonally and they correspond to a face-vertex pair in the tangency graphs. Otherwise, their interiors are disjoint. 
    \item $B \cup \hat{B}$ has at most one accumulation point.
\end{enumerate}
\end{definition}

\noindent Note that the roles of $B$ and $\hat{B}$ are interchangeable in this definition. An example of a base and dual configuration pair is shown in Figure \ref{fig:general}.

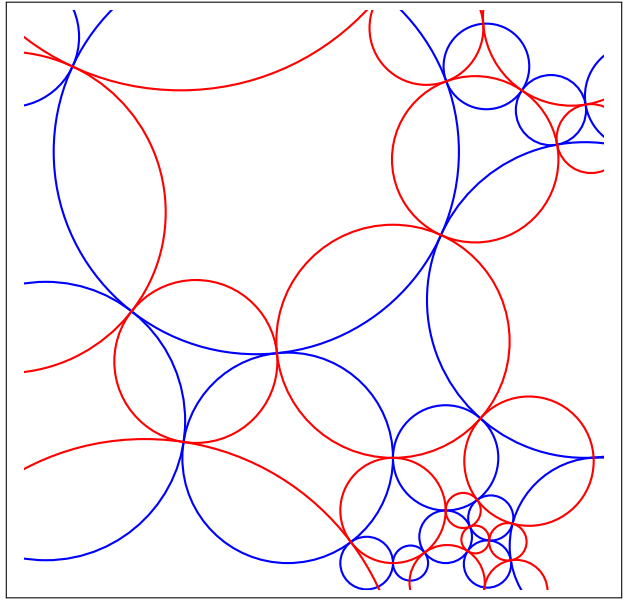
\begin{figure}[H]
\fbox{
\begin{tikzpicture}[scale=.7]
\clip (-8, -2.5) -- (3, -2.5) --(3, 8.5) -- (-8, 8.5);

\draw[thick, blue] (2.90215, -1.67837) circle(1.68539); 
\draw[thick, blue] (-1.5, -2.) circle(0.5); 
\draw[thick, blue] (2., 6.60936) circle(0.666667); 
\draw[thick, blue] (0., 0.) circle(1.); 
\draw[thick, blue] (-3.59569, 5.81749) circle(3.84791); 
\draw[thick, blue] (0.774475, 7.43953) circle(0.81357); 
\draw[thick, blue] (0., -1.5) circle(0.5);
\draw[thick, blue] (-3., 0.) circle(2.); 
\draw[thick, blue] (3.64568, 6.873) circle(1.); 
\draw[thick, blue] (0.857143, -1.14286) circle(0.428571); 
\draw[thick, blue] (-8.29483, 8.) circle(1.33333); 
\draw[thick, blue] (2.64575, 3.) circle(3.); 
\draw[thick, blue] (-7.59275, 0.699762) circle(2.64575); 
\draw[thick, blue] (-0.666667, -2.) circle(0.333333); 
\draw[thick, blue] (0.795004, -2.01821) circle(0.448987);

\draw[thick, red] (2.39202, 8.37676) circle(1.68314);
\draw[thick, red] (2.76183,   6.0684) circle(0.654654);
\draw[thick, red] (0.561685,   5.67477) circle(1.58043);
\draw[thick, red] (1.58301, -0.0627461) circle(1.22876);
\draw[thick, red] (-1.,   2.21525) circle(-2.21525);
\draw[thick, red] (0.333333, -1.) circle(0.333333);
\draw[thick, red] (-1., -1.) circle(-1.);
\draw[thick, red] (-4.74397, 1.83188) circle(-1.5483);
\draw[thick, red] (-8.37305,   4.66627) circle(-3.05648);
\draw[thick, red] (-5.71414, -4.5356) circle(-4.89267);
\draw[thick, red] (-0.364423,   8.16993) circle(1.08105);
\draw[thick, red] (1.18162, -1.59554) circle(0.35572);
\draw[thick, red] (0.563183, -1.55164) circle(0.264276);
\draw[thick, red] (0.0287017, -2.3716) circle(0.714503);
\draw[thick, red] (1.33409, -2.54085) circle(0.601807);
\draw[thick, red] (-5.03541, 11.8525) circle(4.86701);
\end{tikzpicture}
}
\caption{Base and dual circle configurations}
\label{fig:general}
\end{figure}

The tangency graphs $G_{B}$ and $G_{\hat{B}}$ are simple, and by (1) they have natural planar embeddings, placing each vertex at the center of its circle (or choosing an arbitrary interior point as the vertex if the circle is a line). The duality between graphs, with faces of $G_B$ corresponding to vertices of $G_{\hat{B}}$ and vice versa, is defined in this context. The four properties, and especially this duality, imply stronger geometric statements, for example:

\begin{proposition} \label{tangencypoint}
Suppose that two circles $c_1, c_2 \in B$ are tangent at a point $x \in \hat{\C}$. Then no other circles from $B$ go through this point. Moreover, if $x$ is not the accumulation point, then it is also a point of tangency for exactly two dual circles $d_1, d_2 \in \hat{B}$, which intersect $c_1, c_2$ orthogonally.
\end{proposition}
\begin{proof}
First, there cannot be any additional circle in $B$ tangent to $c_1, c_2$ at $x$ as this would violate the disjoint interiors property (1). Next, suppose that $x$ is not the accumulation point, and consider the edge connecting $c_1, c_2$ in $G_B$. This corresponds to an edge connecting some pair of tangent dual circles $d_1, d_2 \in G_{\hat{B}}$, which intersect $c_1, c_2$ orthogonally. After applying a M\"{o}bius transformation, we may assume that $c_1$ and $c_2$ are parallel horizontal lines which are tangent at the point $x=\infty$. Then it is clear that $d_1, d_2$ must be parallel vertical lines which are tangent at the same point $x$. Again, there cannot be any additional circle in $\hat{B}$ tangent to $d_1, d_2$ at $x$ as this would violate property (1).
\end{proof}

\begin{proposition} \label{ringed}
Each circle in $c \in B$ is orthogonal to at least three dual circles in $\hat{B}$. These dual circles can be labeled by elements of $\Z$ or $\Z/n\Z$ so that consecutive circles are tangent.
\end{proposition}
\begin{proof}
The vertex corresponding to $c$ in $G_{B}$ must be connected to the rest of the graph by at least three edges; if it were connected by two edges, the dual graph would have a double edge, and if it were connected by one edge, then the dual graph would have a loop. Thus the vertex is incident to at least three faces. These faces must correspond to distinct dual circles because each dual circle is uniquely determined by two points of intersection with $c$. Then the rest of the statement follows from the fact that the faces incident to a vertex in a planar graph can be ordered cyclically so that any two consecutive faces are adjacent. 
\end{proof}
We will describe the situation of Proposition \ref{ringed} by saying that $c$ is ringed by circles from $\hat{B}$. One subtlety in the proof is the possibility that $c$ could be ringed by infinitely many circles. This can happen if the unique accumulation point lies on $c$, and in this case, the orthogonal circles can be labeled by the integers so that any two consecutive circles are tangent. 

\begin{proposition} \label{cover}
The circles in $B$, $\hat{B}$, and their interiors cover all points in $\hat{\C}$ other than the accumulation point.
\end{proposition}
\noindent The accumulation point may or may not be covered, depending on whether or not it lies on a circle.
\begin{proof}
It suffices to show that all points on one face in the planar embedding of $G_B$ are covered. After a M\"{o}bius transformation, we may assume that this face does not contain the point $\infty$, and that none of the generalized circles in $B$ surrounding it are lines. By Proposition \ref{ringed}, the face is a simply connected polygon (possibly with infinitely many sides), with vertices at the centers of circles in $B$. Each edge goes through a point of tangency between two circles in $B$, and is orthogonal to both circles. Then by Proposition \ref{tangencypoint}, the dual circle in $\hat{B}$ corresponding to the face is tangent to each edge, i.e. it is inscribed in the polygon. Removing the dual circle and its interior leaves one connected component for each vertex, and each connected component is contained in the interior of the corresponding circle. 
\end{proof}

A graph is is said to be $n$-connected if after removing any $n-1$ vertices and their adjacent edges, the graph remains connected. It is said to be $n$-edge connected if after removing any $n-1$ edges, the graph remains connected. Because $G_{B}$ and $G_{\hat{B}}$ are a pair of dual simple planar graphs, they are necessarily 3-edge connected. If removing an edge could disconnect $G_B$, then $G_{\hat{B}}$ would have to contain a loop, and if removing two edges could disconnect $G_B$, then $G_{\hat{B}}$ would have to contain a double edge. In fact, we can make an even stronger connectedness statement:

\begin{proposition} \label{3Connected}
The tangency graphs $G_{B}$ and $G_{\hat{B}}$ are $3$-connected.
\end{proposition}
\begin{proof}
It suffices to show that if one or two vertices and their adjacent edges are removed from $G_{B}$, then all faces in the planar embedding of $G_{B}$ remain simply connected. This implies that $G_{B}$ remains connected. 

If a single vertex $c_1$ is removed, then the faces incident to that vertex are all identified. By Proposition \ref{ringed}, these faces correspond to the ring of dual circles around $c_1$. Because these faces are all distinct and simply connected, gluing them at the vertex and along their common edges results in a larger simply connected face.

If two non-adjacent vertices $c_1, c_2$ are removed, then at each vertex the incident faces are identified. If $c_1, c_2$ have no face in common, then the argument is exactly the same as above. If they have a face in common, there can only be one such face. Indeed, after applying a M\"{o}bius transformation, we may assume that $c_1, c_2$ are concentric circles centered at the origin. Any dual circle orthogonal to both must be a line through the origin. By the disjoint interiors property (1), $\hat{B}$ can contain at most one such line. In this case, removing $c_1$ and $c_2$ creates one new larger face, the union of all faces incident to $c_1$ or $c_2$. Because the union of faces incident to $c_1$ is simply connected, the union of faces incident to $c_2$ is simply connected, and they overlap in a unique simply connected face, the new face must be simply connected.

Finally, if two adjacent vertices $c_1, c_2$ are removed, by Proposition \ref{tangencypoint}, there are exactly two faces incident to these two vertices, meeting along the edge from $c_1$ to $c_2$. The union of these two faces is simply connected. Removing $c_1$ and $c_2$ creates one new larger face, the union of all faces incident to $c_1$ or $c_2$. Because the union of faces incident to $c_1$ is simply connected, the union of faces incident to $c_2$ is simply connected, and they overlap in a simply connected union of two faces, the new face must be simply connected. 
\end{proof}

By the compactness of $\hat{\C}$, $B$ is finite if an only if it has no accumulation point. In this case $G_B$ is a finite 3-connected simple planar graph, so by Steinitz's theorem, it is the graph of a polyhedron.

One further property we might ask for in base and dual configurations is periodicity:
\begin{definition} 
For $n=1, \, 2$, $B$ is $n$-periodic, i.e. periodic under an $n$-dimensional lattice, if there exist $v_1, \dots, v_n \in \C$, linearly independent over $\R$, such that $B + v_1 = \dots = B + v_n = B$.
\end{definition}
\noindent If $B$ is $1$-periodic or $2$-periodic, it necessarily has an accumulation point at $\infty$. The tangency graph of a $2$-periodic $B$ is a $2$-periodic tiling of the plane. 

We are now ready to define the class of circle packings that we will study. For any circle $d$, let $\sigma_d$ denote the reflection across $d$, a M\"{o}bius transformation. 
\begin{definition} \label{packing}
The packing $\mathscr{P}$ is the orbit of $B$ under the group generated by reflections $\sigma_d$ across circles $d \in \hat{B}$. The dual packing $\hat{\mathscr{P}}$ is the orbit of $\hat{B}$ under the same group.

The superpacking is the orbit of $B$ under the group generated by reflections across circles in $B$ and $\hat{B}.$
\end{definition}

Note that under this definition, the dual packing and superpacking contain two oppositely-oriented copies of each circle. It is possible to assign a single orientation to circles in the dual packing and the superpacking in a consistent way, but since we will not use the orientations of these circles in an important way, we do not pursue this. We will only be concerned with the orientations of the circles in $\mathscr{P}$.

For each packing $\mathscr{P}$, we define the following symmetry groups:
\begin{enumerate}
    \item $\Gamma=\textsc{Sym}(\mathscr{P},\hat{\mathscr{P}})$: the group of M\"{o}bius transformations that preserve both the packing and the dual packing;
    \item $\Gamma_1=\langle \sigma_d \, : \, d \in \hat{B} \rangle$:  the group generated by reflections across the dual circles;
    \item $\Gamma_2=\textsc{Sym}(B,\hat{B})$: the group of M\"{o}bius transformations that preserve both the base configuration and the dual configuration.\end{enumerate}

Some fundamental geometric properties of $\mathscr{P}$ can be deduced directly from the definition.

\begin{proposition}\label{dualOrInside}
Every circle in $\mathscr{P}$ is either a base circle in $B$ or inside some dual circle, and every circle in $\hat{\mathscr{P}}$ is either a dual circle in $\hat{B}$ or inside some dual circle.
\end{proposition}
Note that when we say circle $c_1$ is inside circle $c_2$, we only mean that $c_1$ is contained in the union of $c_2$ and its interior, not necessarily that the two interiors are nested.

\begin{proof}
By definition, each circle in $\mathscr{P}$ is in $B$ or it is $\sigma_{d_1}\cdots\sigma_{d_k}(c)$ for some $c \in B$, $d_1, \ldots d_k \in \hat{B}$, with consecutive $d_i$ distinct. We may assume that $c$ is not orthogonal to $d_k$, because then $d_k$ could be dropped from this expression. Then by property (3), $c$ is outside $d_k$, so $\sigma_{d_k}(c)$ is inside $d_k$. By property (1), this implies that $\sigma_{d_k}(c)$ is outside $d_{k-1}$, so $\sigma_{d_{k-1}}\sigma_{d_k}(c)$ is inside $d_{k-1}$. Repeating the argument inductively, we see that $\sigma_{d_1}\cdots\sigma_{d_k}(d)$ is inside $d_1$. The proof of the second statement is similar.
\end{proof}

\begin{proposition} \label{disjointinteriors}
The circles in $\mathscr{P}$ are pairwise disjoint or tangent, with disjoint interiors.
\end{proposition}
\begin{proof}
Suppose that we have two circles in $\mathscr{P}$ with overlapping interiors. After applying a M\"{o}bius transformation, we may assume that one of the circles, $c_1$ is in $B$.  We may write the other circle as $\sigma_{d_1} \cdots \sigma_{d_k}(c_2)$ for $c_2 \in B$, $d_1, \ldots d_k \in \hat{B}$, with consecutive $d_i$ distinct. If $d_k$ is orthogonal to $c_2$, then we can shorten this expression to $\sigma_{d_1} \cdots \sigma_{d_{k-1}}(c_2)$. Otherwise, by property (3) of Definition \ref{basedual}, $d_k$ and $c_2$ have disjoint interiors, and after inversion, the interior of $\sigma_{d_k}(c_2)$ is contained inside the interior of $d_k$. After the remaining inversions, by property (1), the interior of $\sigma_{d_1} \cdots \sigma_{d_k}(c_2)$ is contained inside the interior of $d_1$. By property (3), in order for $c_1$ to intersect $\sigma_{d_1} \cdots \sigma_{d_k}(c_2)$, $c_1$ must be orthogonal to $d_1$. Then inverting across $d_1$, we find that $c_1$ also intersects $\sigma_{d_2} \cdots \sigma_{d_k}(c_2)$. In either case, we have shortened the string $\sigma_{d_1} \cdots \sigma_{d_k}$. Repeating this process, we eventually find two circles $c_1, c_2 \in B$ with overlapping interiors, contradicting property (1).
\end{proof}

\begin{proposition} \label{dense}
The interiors of the circles in $\mathscr{P}$ are dense in $\hat{\C}$.
\end{proposition}
\begin{proof}
For $x\in \hat{\C}$, by Proposition \ref{cover}, either $x$ is the accumulation point, $x$ is in some circle in $B$, or $x$ is in some circle in $\hat{B}$. It follows that either $x$ is in the closure of the interiors of the circles in $B$ or $x$ is in the interior of some circle in $\hat{B}$. If $x$ is in the interior of a dual circle $d_1$, reflect it across $d_1$. Then either $\sigma_{d_1}(x)$ is in the closure of the interiors of the circles in $B$, which means that $x$ is in the closure of the interiors of the circles in $\sigma_{d_1}(B)$, or $\sigma_{d_1}(x)$ is in the interior of some circle $d_2\in\hat{B}$, which means $x$ is in the interior of $\sigma_{d_1}(d_2)$. In the latter case, reflect $\sigma_{d_1}(x)$ across $d_2$ and repeat the process. If, at some step, we find that $x$ is in the closure of the interiors of the circles in $\sigma_{d_1}\cdots\sigma_{d_k}(B)$, then since these are circles in $\mathscr{P}$, $x$ is in the closure of the interiors of the circles in $\mathscr{P}$ as desired. 

Otherwise, this construction produces an infinite sequence
$$\tilde{d}_1 = d_1, \, \tilde{d}_2 = \sigma_{d_1}(d_2), \, \tilde{d}_3=\sigma_{d_1}\sigma_{d_2}(d_3), \ldots$$
with consecutive $d_i$ distinct, of nested circles in $\hat{\mathscr{P}}$ whose interiors contain $x$. We will show that these circles converge to $x$ in $\hat{\C}$, in the sense that they eventually lie within any open neighborhood of $x$. After applying a M\"{o}bius transformation, we may assume that $\tilde{d}_1$ is not a line and is oriented inward so that its interior does not contain $\infty$. Since the circles $\tilde{d}_k$ are nested, they all have these properties.

We have that
$$\tilde{d}_{k+1} =\sigma_{\tilde{d}_{k}}(\sigma_{d_1}\cdots \sigma_{d_{k-1}}(d_{k+1}))$$
by Lemma \ref{image_of_dual} below. If $d_{k+1} \neq d_{k-1}$, then $\sigma_{d_1}\cdots \sigma_{d_{k-1}}(d_{k+1})$ and $\tilde{d}_k$ are both contained inside $\tilde{d}_{k-1}$, so neither one contains infinity. Thus, the reflection of $\sigma_{d_1}\cdots \sigma_{d_{k-1}}(d_{k+1})$ across $\tilde{d}_k$ does not contain the center of $\tilde{d}_k$, so the radius of $\tilde{d}_{k+1}$ is at most half the radius of $\tilde{d}_k$. Therefore, if $d_{k+1} \neq d_{k-1}$ for infinitely many values of $k$, then the radii must approach $0$, so the circles must approach $x$. 

If $d_{k+1} = d_{k-1}$, then $\tilde{d}_{k+1} =\sigma_{\tilde{d}_{k}}(\tilde{d}_{k-1})$ (with the orientation reversed). If this holds for all but finitely many values of $k$, then there exists some $K \in \N$ such that for all $k \geq K$, $\tilde{d}_{k+1} =\sigma_{\tilde{d}_{k}}(\tilde{d}_{k-1})$. Suppose that $\tilde{d}_{K}$ and $\tilde{d}_{K+1}$ are disjoint. Then, after another M\"{o}bius transformation, we may assume that $\tilde{d}_{K}$ has radius $1$, $\tilde{d}_{K+1}$ has radius $r<1$, and they are concentric. An inductive argument shows that the radius of $\tilde{d}_{K+k}$ is $r^k$. These radii approach $0$, so the circles must approach $x$.

On the other hand, suppose that $\tilde{d}_{K}$ and $\tilde{d}_{K+1}$ are tangent. After a M\"{o}bius transformation, we may assume that $\tilde{d}_{K}$ has radius $1$, and $\tilde{d}_{K+1}$ has radius $r<1$. In this case, an inductive argument shows that the radius of $\tilde{d}_{K+k}$ is $\frac{1}{1+k(1/r-1)}$. These radii approach $0$, so the circles must approach $x$.

Since the circles $\tilde{d}_k \in \hat{\mathscr{P}}$ approach $x$, there is a corresponding sequence of circles $\tilde{c}_k \in \mathscr{P}$, with $\tilde{c}_k$ orthogonal to $\tilde{d}_k$, such that the interiors of the circles $\tilde{c}_k$ come arbitrarily close to $x$.
\end{proof}

The residual set of $\mathscr{P}$ is the set of points not in the interior of any circle. 

\subsection{Relation to other Packing Definitions}

In \cite{KontorovichNakamura} and \cite{KapovichKontorovich}, definitions are given for the related notions of Kleinian, crystallographic, and polyhedral packings. To compare our definition to these, we give sufficient conditions for our packings to be polyhedral, crystallographic, or Kleinian.

The definition of a polyhedral packing coincides with our definition, with the stricter additional assumption that $G_B$ and $G_{\hat{B}}$ are the graphs of a convex polyhedron and its dual. By the Koebe-Andreev-Thurston theorem, every polyhedron gives rise to a polyhedral packing. Proposition \ref{3Connected} and Steinitz's theorem imply the following:
\begin{proposition}\label{polyhedralcondition}
If $B$ and $\hat{B}$ are finite circle configurations satisfying the conditions of Definition \ref{basedual}, then they give rise to a polyhedral circle packing $\mathscr{P}$. 
\end{proposition}

The definitions of crystallographic and Kleinian packings are more general, and relate the packing to a discrete group action on a higher-dimensional hyperbolic space. The two-sphere $\hat{\C}$ is identified with the boundary of three-dimensional hyperbolic space $\mathbb{H}^3$, and M\"{o}bius transformations are viewed as isometries of $\mathbb{H}^3$. The definitions of crystallographic and Kleinian circle packings are as follows:
\begin{definition}
Let $\mathscr{P}$ denote a collection of circles whose interiors are disjoint and dense in $\hat{\C}$. $\mathscr{P}$ is a crystallographic packing if its residual set is the limit set of a geometrically finite reflection group of isometries of $\mathbb{H}^3$. It is Kleinian if its residual set is the limit set of any geometrically finite group of isometries of $\mathbb{H}^3$.
\end{definition}
\noindent All polyhedral packings are crystallographic, and all crystallographic packings are Kleinian. 

The following theorem gives sufficient conditions for a packing $\mathscr{P}$ to be crystallographic or Kleinian.
\begin{theorem} \label{Kleiniancondition}
A packing $\mathscr{P}$ satisfying Definition \ref{packing} is Kleinian if any of the following conditions hold:
\begin{enumerate}
\item $B \cup \hat{B}$ is finite.
\item After applying some M\"{o}bius transformation, $B\cup\hat{B}$ is a strip configuration (i.e. it contains two parallel lines) and $\Gamma_2$ contains a translation.
\item After applying some M\"{o}bius transformation, $\Gamma_2$ contains two linearly independent translations.
\end{enumerate}
Furthermore, $\mathscr{P}$ is crystallographic if any of the above conditions hold, with $\Gamma_2$ replaced by the maximal reflective subgroup of $\Gamma_2$.
\end{theorem}
\begin{proof}
By Propositions \ref{disjointinteriors}, \ref{dense}, the interiors of circles in $\mathscr{P}$ are disjoint and dense in $\hat{\C}$. We will show that each condition (1)-(3) implies that $\Gamma$, viewed as a group of isometries of $\mathbb{H}^3$, is geometrically finite. This is essentially a consequence of Theorems \ref{gamma2}, \ref{semidirectproduct}. Up to finite index, we have the following fundamental domains for $\Gamma_2$: in case (1) all of $\hat{\C}$, in case (2) a strip, and in case (3) a compact parallelogram. Each condition implies that this fundamental domain only intersects finitely many dual circles in $\hat{B}$. Using the half-space model of $\mathbb{H}^3$, we may define a half-plane for each wall of the fundamental domain for $\Gamma_2$, and a hemisphere for each dual circle in this fundamental domain. These walls bound a geometrically finite fundamental domain for $\Gamma$. The limit set of $\Gamma$ is the residual set of $\mathscr{P}$. Thus the packing is Kleinian. The proof in the crystallographic case works similarly. 
\end{proof}

Our construction gives rise to packings which are polyhedral (in fact, all polyhedral packings), packings which are crystallographic but not polyhedral, packings which are Kleinian but not crystallographic, and packings which are none of the above. In Section 6, the wallpaper groups generated by reflections give rise to crystallographic packings; the wallpaper groups containing translations but no reflections give rise to Kleinian packings. And the same refinement method in the proof of Theorem \ref{wallpaper} can produce aperiodic packings, which are not Kleinian. These examples can be constructed to have integrality or superintegrality properties. 

Note that Proposition \ref{polyhedralcondition} and Theorem \ref{Kleiniancondition} give sufficient but not necessary conditions for $\mathscr{P}$ to be polyhedral, crystallographic, or Kleinian. Their proofs involve the structure of the group $\Gamma=\textsc{Sym}(\mathscr{P},\hat{\mathscr{P}})$. But $\mathscr{P}$ may have additional symmetry not detected by $\Gamma$; in general, $\textsc{Sym}(\mathscr{P},\hat{\mathscr{P}}) \neq \textsc{Sym}(\mathscr{P})$. We might make different choices of base and dual configuration $B$, $\hat{B}$, which give rise to the same packing $\mathscr{P}$ but a different dual packing $\hat{\mathscr{P}}$, and thus a different $\Gamma$. For example, working with the classical Apollonian packing, we could select any collection of circles defining a 3-connected subgraph of the full tangency graph as $B$. With a nonstandard choice of base and dual configuration, we would find a smaller symmetry group $\Gamma$, and the packing would not be immediately identifiable as polyhedral, crystallographic, or Kleinian. 

\section{Existence and Uniqueness of Packings and a Generalized Koebe-Andreev-Thurston Theorem}

In this section, we recall different versions of the Koebe-Andreev-Thurston theorem that imply the existence and uniqueness up to conformal automorphism of our packings in many cases. We conjecture a generalization which would imply existence and uniqueness in all cases. 

To state the problem precisely: let $G$ and $\hat{G}$ be a pair of simple 3-connected plane graphs corresponding to the 1-skeleton of a cellular decomposition of the sphere $\hat{\C}$ or the plane $\C$, and the 1-skeleton of the dual cellular decomposition. Note that a cellular decomposition of the plane can be viewed as a cellular decomposition of the sphere with a unique accumulation point at $\infty$. Do there exist circle configurations $B$ and $\hat{B}$, satisfying the conditions of Definition \ref{basedual}, such that $G_B \cong G$ and $G_{\hat{B}} \cong \hat{G}$? The $\cong$ symbol here means an isomorphism of graphs and of the associated cellular decompositions. Moreover, are the configurations $B$ and $\hat{B}$ unique up to conformal automorphism of $\hat{\C}$ or $\C$? A conformal automorphism of $\hat{\C}$ is a M\"{o}bius transformation; a conformal automorphism of $\C$ is a M\"{o}bius transformation which fixes infinity, i.e. a similarity $z \mapsto az+b$ or $z \mapsto a \bar{z}+b$. 

The Koebe-Andreev-Thurston theorem answers these questions in the affirmative when $G$ and $\hat{G}$ are finite graphs. A version of the theorem closely aligned with this article appears in \cite{BrightwellScheinerman}. Restated in our language, their Theorem 6 is as follows:

\begin{theorem} \label{KAT1}
Let $G$ and $\hat{G}$ be a pair of finite, simple, 3-connected plane graphs, corresponding to a cellular decomposition of $\hat{\C}$ and its dual. Then there exist circle configurations $B$, $\hat{B}$, satisfying Definition \ref{basedual}, such that $G_B \cong G$ and $G_{\hat{B}} \cong \hat{G}$, and these configurations are unique up to M\"{o}bius transformation.
\end{theorem}

As discussed in the introduction, this theorem has been extended to some infinite graphs. Theorem 4.3, the ``Discrete Uniformization Theorem'' in \cite{Stephenson} implies the following:

\begin{theorem} \label{KAT2}
Let $G$ be an infinite simple, 3-connected plane graph, corresponding to a triangulation of $\C$. Then there exist circle configurations $B$, $\hat{B}$, satisfying Definition \ref{basedual}, such that $G_B \cong G$, and these configurations are unique up to similarity.
\end{theorem}

\noindent Stephenson's statement of the theorem includes triangulations of other surfaces as well. The restriction to triangulations, not general cellular decompositions, seems to be a convenient simplification rather than an essential restriction. In this case, the existence and uniqueness of $B$ immediately imply the existence and uniqueness of $\hat{B}$. But if $G$ is not a triangulation, then $B$ and $\hat{B}$ must be constructed together, and both are needed to ensure uniqueness. 

Given these two versions of the Koebe-Andreev-Thurston theorem, it is natural to conjecture the following common extension:

\begin{conjecture}\label{KATconjecture}
Let $G$ and $\hat{G}$ be a pair of simple, 3-connected plane graphs, corresponding to a cellular decomposition of $\C$ and its dual. Then there exist circle configurations $B$, $\hat{B}$, satisfying Definition \ref{basedual}, such that $G_B \cong G$ and $G_{\hat{B}} \cong \hat{G}$, and these configurations are unique up to similarity.
\end{conjecture}

\noindent This would imply the existence and uniqueness of the whole class of packings we study. We hope to prove this conjecture in future work, following the methods of \cite{Stephenson}. As further evidence for the conjecture, we remark that several special cases have been studied carefully in the literature. When $G$ is the triangular lattice, the uniqueness of the associated circle configuration is a crucial step in Rodin and Sullivan's celebrated proof of the convergence of circle packings to the Riemann mapping \cite[Appendix 1]{RodinSullivan}. This means that the triangular and hexagonal packings studied in Section 5 are unique up to similarity. When $G$ is the square lattice, the uniqueness of the associated circle configuration is the main theorem of \cite{Schramm:SquareGrid}. So the square packing in Section 5 is also unique up to similarity. 

It should also be possible to extend Conjecture \ref{KATconjecture} to the hyperbolic plane and other surfaces, but we have not investigated this. 

\section{Group Structure Theorems}

In this section, we give a complete algebraic description of the symmetry groups associated to packings $\mathscr{P}$. We begin with the structure of the groups $\Gamma_1$ and $\Gamma_2$. Then, via an examination of the action on $\mathscr{P}$, we show that $\Gamma=\Gamma_1 \rtimes \Gamma_2$. 

\begin{theorem}\label{gamma1}
$\Gamma_1$ is a free Coxeter group generated by $\sigma_d$ for $d \in \hat{B}$, where the only relations are $\sigma_d^2=1$. 
\end{theorem}
\begin{proof}
We must show that the relations $\sigma_d^2=1$ are the only ones. Choose a point $x \in \hat{\C}$, outside of every circle in $\hat{B}$. Suppose that a string $\sigma_{d_1} \cdots \sigma_{d_k} \in \Gamma_1$, with consecutive $d_i$ distinct, is applied to $x$. Since $x$ is outside of $d_k$, $\sigma_{d_k}(x)$ is in the interior of $d_k$. Repeating this process, by property (1) from Definition \ref{basedual}, we find that $\sigma_{d_1} \cdots \sigma_{d_k}(x)$ is in the interior of $d_1$. Thus it is not equal to $x$, and $\sigma_{d_1} \cdots \sigma_{d_k}$ is not the identity.
\end{proof}

\begin{theorem} \label{gamma2}
If $B$ is finite, then $\Gamma_2$ is the group of symmetries of a polyhedron. If $B$ is infinite, then $\Gamma_2$ is conjugate to a discrete group of isometries of the plane: a cyclic group, dihedral group, frieze group, or wallpaper group. 
\end{theorem}
\noindent A frieze group is a discrete group of isometries of the plane which contains translations in one direction; a wallpaper group is a discrete group of isometries of the plane which contains translations in two linearly independent directions.
\begin{proof}
If $B$ is finite, then we have shown that $G_B$ is the graph of a polyhedron. By the Koebe-Andreev-Thurston theorem, the circle configuration $B$ with this graph is unique up to M\"{o}bius transformations. Moreover, any graph automorphism of $G_B$ gives rise to a permutation of the circles in $G_B$, which must be realized by a M\"{o}bius transformation in $\Gamma_2$. Conversely, any element of $\Gamma_2$ determines a graph automorphism of $G_B$. So $\Gamma_2 \equiv \textsc{Aut}(G_B)$. There is a three-dimensional realization of the polyhedron $G_B$, called the canonical embedding, with all edges tangent to the unit sphere, such that every automorphism of $G_B$ is realized as a rigid motion of this polyhedron preserving the sphere \cite[Thm. 4.13]{Ziegler}. Thus $\Gamma_2$ is the group of symmetries of a polyhedron, or a finite group of isometries of the sphere.

If $B$ is infinite, then we may apply a M\"{o}bius transformation to place the unique accumulation point of $B$ at $\infty$. All symmetries of $B$ must map the accumulation point to itself, so they must have the form $f(z)= az+b$ or $f(z)= a \bar{z}+b$. If $|a| \neq 1$, then such a map has a fixed point in $\C$, which will be an attracting fixed point for $f$ or $f^{-1}$. This produces an additional accumulation point for $B$, a contradiction. Thus $\Gamma_2$ consists of maps $f(z)= az+b$ or $f(z)= a \bar{z}+b$ with $|a|=1$, which are isometries of the plane $\C$. Again because $B$ has no accumulation points in $\C$, $\Gamma_2$ must be a discrete group of isometries of $\C$. The rest of the theorem follows from the classification of these groups, see \cite{Coxeter}.
\end{proof} 

In Section 6, we will show that each of the above possibilities for $\Gamma_2$ is in fact realized by an appropriate choice of $B$.

In order to understand the interactions of $\Gamma_1$ and $\Gamma_2$, we need further geometric information about their action on packings. Any circle in $\mathscr{P}$ is $c=\sigma_{d_1}\cdots\sigma_{d_k}(c_0)$ for some $c_0 \in B$, $d_1, \ldots d_k \in \hat{B}$. Define the height $\mathrm{ht}(c)$ as the minimum $k$ for which such an expression exists. A circle has height $0$ if and only if it is in $B$. 

\begin{lemma} \label{heightdecrease}
Suppose $c \in \mathscr{P}$ has $\mathrm{ht}(c)>0$. For $d \in \hat{B}$, $\mathrm{ht}(\sigma_d(c))<\mathrm{ht}(c)$ if and only if $c$ is inside $d$. 
\end{lemma}
\begin{proof}
Say that $c=\sigma_{d_1}\cdots\sigma_{d_k}(c_0)$ and that this expression is minimal. Then $c_0$ is not orthogonal to $d_k$, so it is outside $d_k$. Thus $\sigma_{d_k}(c_0)$ is inside $d_k$, and hence outside $d_{k-1}$. Repeating inductively, we conclude that $c$ is inside $d_1$. Since the circles of $\hat{B}$ have disjoint interiors, $d_1$ is the unique circle in $\hat{B}$ whose interior contains $c$. If we reflect through $d_1$, the height of $c$ will be lowered. 

On the other hand, if we reflect through any other circle $d \in \hat{B}$, $\sigma_d(c)$ will be inside $d$. Then $\mathrm{ht}(c) = \mathrm{ht}(\sigma_d\sigma_d(c))<\mathrm{ht}(\sigma_d(c))$, so the height of $c$ will be raised.
\end{proof}

\begin{lemma}\label{relativePosTrans}
For any $g\in\Gamma$ and $d\in\hat{B}$, $g(B)$ is either completely inside $d$ or completely outside $d$. (Some circles of $g(B)$ are allowed to be orthogonal to $d$ in either case.)
\end{lemma}

\begin{proof}
For convenience in this proof, we consider a circle orthogonal to $d$ as being both inside and outside $d$. Since $d$ is a dual circle and $g$ preserves $\hat{\mathscr{P}}$, we know that $g^{-1}(d)\in\hat{\mathscr{P}}$. By Proposition \ref{dualOrInside}, 
we know that for some $\Tilde{d}\in\hat{B}$,  $g^{-1}(d)$ is either $\Tilde{d}$ 
or inside $\Tilde{d}$. All circles in $B$ are completely outside $\Tilde{d}$ by definition, so they are outside $g^{-1}(d)$ as well. Applying $g$ to both $B$ and $g^{-1}(d)$, 
we have that all circles in $g(B)$ are completely inside or completely outside $d$.
\end{proof}

\begin{lemma}\label{image_of_dual}
Let $g$ be a M\"{o}bius transformation, $d$ be a circle, and $\sigma_d$ be the reflection across $d$. We have $g\sigma_d g^{-1} = \sigma_{g(d)}$.
\end{lemma}

\begin{proof}
The map $g \sigma_d g^{-1} \sigma_{g(d)}$ is a holomorphic M\"{o}bius transformation which fixes the circle $g(d)$, so it is the identity. Thus $g\sigma_d g^{-1} = \sigma_{g(d)}$.
\end{proof}

\begin{proposition}\label{Gamma}
For a circle packing $\mathscr{P}$ with base configuration $B$ and dual configuration $\hat{B}$, we have $\Gamma=\left\langle \Gamma_{1},\Gamma_{2}\right\rangle$.
\end{proposition}

\begin{proof}
By definition, we have $\Gamma_{1}\leq\Gamma$. For any $g\in\Gamma_{2}$, 
we know $g$ preserves $\hat{\mathscr{P}}$ because it preserves $\hat{B}$, 
and we also know $g$ preserves $\mathscr{P}$ because it preserves both $B$ and $\hat{B}$. Hence, we have $\Gamma_{2}\leq\Gamma$ and thus $\left\langle \Gamma_{1},\Gamma_{2}\right\rangle \subseteq\Gamma$. 

We now want to show that $\Gamma \subseteq \langle \Gamma_1, \Gamma_2 \rangle$. By Lemma \ref{relativePosTrans}, for any $g\in\Gamma$, we know $g(B)$ is either outside all dual circles or is completely inside some dual circle. 

Case 1: If $g(B)$ is outside all dual circles, then $g(B)\subseteq B$. Let $G_{B}$ be the tangency graph of $B$, and let $G_{g(B)}$ be the subgraph of $G_{B}$ that is also the tangency graph of $g(B)$. The faces of $G_{g(B)}$ are either faces of $G_B$ or unions of these faces. We will show that each face of $G_{g(B)}$ is a face of $G_B$, implying that $g(B)=B$. Notice that every face in $G_{g(B)}$ corresponds to a circle $g(\Tilde{d})\in\hat{\mathscr{P}}$ for some $\Tilde{d}\in\hat{B}$. Since $g(\Tilde{d})$ is ringed by circles from $g(B)$ and thus from $B$, we know $g(\Tilde{d})$ cannot be inside any dual circle. By Proposition \ref{dualOrInside}, then, we know that $g(\Tilde{d})$ is a dual circle, which implies that $g(B)=B$. Because $g$ preserves $G_{B}$, we know $g$ also preserves the dual graph of $G_{B}$, which is the tangency graph of $g(\hat{B})$. Let $G_{\hat{B}}$ be the dual graph of $G_{B}$. Since there is a unique circle orthogonal to a ring of circles, every vertex in $G_{\hat{B}}$ must correspond to a circle in $\hat{B}$. Therefore, we have $g(\hat{B})=\hat{B}$ and thus $g\in\Gamma_2$.

Case 2: If $g(B)$ is completely inside some dual circle $d_1 \in\hat{B}$, reflect across $d_1$ and apply Lemma \ref{relativePosTrans} to the configuration $\sigma_{d_1}g(B)$. If this configuration is outside all dual circles, then by Case 1, we conclude that $\sigma_{d_1}g\in\Gamma_2$. Otherwise, it lies inside some other dual circle $d_2$, and we can reflect across $d_2$ and repeat the argument. It suffices to show that this process eventually terminates, i.e. that $\sigma_{d_k}\cdots \sigma_{d_1}g(B)$ is outside all the dual circles for some $k$.

Consider a finite set $S$ of circles in $g(B)$ such that no circle in $\hat{\C}$ is orthogonal to all of them. Such a set can be obtained starting from three circles ringing a common dual circle, and then choosing a fourth which is not part of this ring. At each step of reflecting through a dual circle, the height of each circle in $S$ decreases or stays constant by Lemma \ref{heightdecrease}. Moreover, the latter possibility can only occur if a circle is orthogonal to the dual, so at least one circle's height decreases at each step. The process terminates when all the circles in $S$ reach height $0$, so they are mapped to $B$. When this occurs, the circles in $S$ are all orthogonal to or outside each circle in $\hat{B}$, and at least one is outside each dual circle. By Lemma \ref{relativePosTrans}, all circles in $g(B)$ are then outside or orthogonal to each dual circle.
\end{proof}

\begin{cor}\label{g1g2}
For any $g\in\Gamma$, there exists $g_1\in\Gamma_1$ and $g_2\in\Gamma_2$ such that $g=g_1g_2$.
\end{cor}

\begin{proof}
By the proof of Theorem \ref{Gamma}, every $g\in\Gamma$ is of the form $g=\sigma_{d_1}\cdots\sigma_{d_k}g_2$ for some sequence of $d_i\in\hat{B}$ and $g_2\in\Gamma_2$. Let $g_1= \sigma_{d_1}\cdots\sigma_{d_k}$. By definition, we know that $g_1\in\Gamma_1$, and $g = g_1g_2$.
\end{proof}

\begin{proposition}\label{normal}
$\Gamma_1$ is a normal subgroup of $\Gamma$.
\end{proposition}

\begin{proof}
By Proposition \ref{Gamma}, it suffices to show that $\Gamma_2$ normalizes $\Gamma_1$. Let $g \in \Gamma_2$, and let $\sigma_d$ be a generator of $\Gamma_1$. By Lemma \ref{image_of_dual}, $g\sigma_d g^{-1} = \sigma_{g(d)}$. Since $g(d) \in \hat{B}$, this is an element of $\Gamma_1$. Thus $\Gamma_2$, and hence $\Gamma$, normalizes $\Gamma_1$. 
\end{proof} 

In general, $\Gamma_2$ is not a normal subgroup of $\Gamma$. 
Let $g_2$ be an element in $\Gamma_2$ that sends some dual circle $d\in\hat{B}$ to a different dual circle in $\hat{B}$. The element $\sigma_{d}g_2\sigma_{d}^{-1}\in\Gamma$ sends $d$ to a circle inside $d$, which means that $\sigma_{d}g_2\sigma_{d}^{-1}$ doesn't preserve $\hat{B}$ and thus $\sigma_{d}g_2\sigma_{d}^{-1}\not\in\Gamma_2$.

\begin{proposition}\label{intersection}
The intersection of $\Gamma_1$ and $\Gamma_2$ is trivial.
\end{proposition}

\begin{proof}
Let $\sigma_{d_1}\cdots \sigma_{d_k} \in\Gamma_1$ with consecutive $d_i$ distinct. Choose a dual circle $d\neq d_k$ in $\hat{B}$. As in the proof of Prop. \ref{dualOrInside}, the map $\sigma_{d_1}\cdots \sigma_{d_k} $ sends $d$ to a circle inside $d_1$, so it doesn't preserve $\hat{B}$ and thus is not in $\Gamma_2$.
\end{proof}

As a a direct result of Proposition \ref{Gamma}, Proposition \ref{normal}, and Proposition \ref{intersection}, we have the following theorem:
\begin{theorem} \label{semidirectproduct}
$\Gamma\cong\Gamma_1\rtimes\Gamma_2$.
\end{theorem}

We conclude this section by sketching some results on the structure of the supergroup. Let $\hat{\Gamma}_1$ denote the group generated by reflections across the base circles $b\in B$. Recall that the superpacking $\mathscr{S}$ is defined as the orbit of $B$ under the group $\langle \Gamma_1, \hat{\Gamma}_1\rangle$. The dual superpacking $\hat{\mathscr{S}}$ is the orbit of $\hat{B}$ under this group. Define $\Gamma_{\mathscr{S}} = \textsc{Sym}(\mathscr{S}, \hat{\mathscr{S}})$, the group of M\"{o}bius transformations which preserve both the superpacking and the dual superpacking. One has the following:

\begin{theorem}
The group $\langle \Gamma_1, \hat{\Gamma}_1\rangle$ is a Coxeter group with generators $\sigma_c$ for $c \in B$, $\sigma_d$ for $d \in \hat{B}$, and relations $\sigma_c^2=1$ for all $c \in B$, $\sigma_d^2=1$ for all $d \in \hat{B}$, and $\sigma_c \sigma_d = \sigma_d \sigma_c$ for all pairs $c \in B$, $d \in \hat{B}$ intersecting orthogonally.
\end{theorem}

\begin{theorem}
$\Gamma_{\mathscr{S}} \cong \langle \Gamma_1, \hat{\Gamma}_1\rangle \rtimes \Gamma_2$. 
\end{theorem}

We omit the proofs because they are similar to previous ones in this section.

\section{Examples}

This section introduces examples of our construction. We call our three main examples the triangular, square, and hexagonal packings. We will focus on the arithmetic properties of these examples: quadratic forms and linear relations satisfied by the curvatures, integrality and superintegrality. A ``typical'' packing satisfying our definition will have little arithmetic interest, but the highly symmetric nature of these examples adds more structure.

Some of these examples have appeared in the literature in other contexts, but their properties have not been explored in detail. The square packing is discussed in \cite[Figure 10.17]{MumfordSeriesWright} as the limit of the $1/n$ cusp groups in Maskit's slice. As discussed in the introduction, the triangular and hexagonal packings appear in Kontorovich and Nakamura's classification of superintegral crystallographic packings. And our examples are closely related to the limit packings in \cite{Ahmed2}. The limit of pyramid packings is the original Apollonian packing, the limit of prism packings is our square packing, and the limit of antiprism packings is our triangular packing. Other families have more complicated limits, which will be briefly discussed at the end of this section. 

Every packing has linear and quadratic forms satisfied by the curvatures, like the Descartes quadratic form for the Apollonian packing. The following definition characterizes packings with number-theoretic structure:
\begin{definition}
A packing $\mathscr{P}$ is integral if every circle in $\mathscr{P}$ has integral curvature. The packing is superintegral if every circle in the superpacking has integral curvature.
\end{definition}
\noindent We will also say that an equivalence class of packings under M\"{o}bius transformations is (super)integral if one packing in the class has this property. 

The main tool to find the linear and quadratic forms, and check (super)integrality, is an inversive coordinate system for oriented generalized circles in $\hat{\C}$. A circle is represented as $(\tilde{b}, b, h_1, h_2)^T \in \R^4$, where $b$ is the signed curvature, $\tilde{b}$ is the curvature after inversion through the unit circle, and $(h_1, h_2)$ are the coordinates of the center, multiplied by the curvature. Every circle satisfies the quadratic equation $h_1^2+h_2^2-b \tilde{b}=1$. The action of M\"{o}bius transformations on generalized circles becomes a linear action preserving the quadratic form in this coordinate system. This setup is well explained in \cite{LagariasMallowsWilks} and in \cite{Kocik2}. The article \cite{Ahmed1} gives a full set of linear and quadratic forms for all polyhedral packings. 

In each of the following examples, we begin with base and dual circle configurations. We check integrality and superintegrality using the inversive coordinate system. We give quadratic and linear relations sufficient to determine the curvatures of all the circles in the packing from a finite set of base circles (in fact, just three). The proofs of these relations are omitted because they are similar to the polyhedral case. We also relate these packings to others with commensurate symmetry groups. 

\subsection{Square Packing}

The base and dual configurations for this packing are shown in Figure \ref{fig:square} (A). The dual configuration is a translation of the base configuration. In these configurations, all the circles can be represented with coordinates $(\tilde{b}, b, h_1, h_2)^T \in \Z^4$, so the packing is superintegral. 

Quadratic and linear relations satisfied by curvatures in the square packing are shown in Figure \ref{fig:square2}. In the formulas, $b_i$ represents the curvature of circle $i$. Any image of one of these configurations under $\Gamma$ will satisfy the same relation. The symmetry group of the square superpacking is commensurate to the symmetry group of the Apollonian superpacking, as illustrated in Figure \ref{fig:square3}. The full packing is shown in Figure \ref{fig:square} (B). 

\begin{figure}[H]

\begin{subfigure}{\textwidth}
\begin{center}
\begin{tikzpicture}[scale = 0.5]
\draw(-2, 0) circle (1cm);
\node at (-2, 0) {1};
\draw(0, 0) circle (1cm);
\node at (0, 0) {2};
\draw(0, -2) circle (1cm);
\node at (0, -2) {3};
\draw(2, -2) circle (1cm);
\node at (2, -2) {4};
\node at (0, -4) {$(b_1-3b_2)^2 + (b_4-3b_3)^2 = 2(b_1+b_2)(b_3+b_4)$};
\end{tikzpicture}
\end{center}
\end{subfigure}

\vspace{.2in}

\begin{subfigure}{.3\textwidth}
\begin{tikzpicture}
\draw(-1, 0) circle (0.5cm);
\node at (-1, 0){1};
\draw(0, 0) circle (0.5cm);
\draw(1, 0) circle (0.5cm);
\node at (1, 0){3};
\draw(0, 1) circle (0.5cm);
\node at (0, 1){2};
\draw(0, -1) circle (0.5cm);
\node at (0, -1){4};
\node at (0,-2) {$b_1 + b_3 = b_2 + b_4$};
\end{tikzpicture} 
\end{subfigure}
\begin{subfigure}{.3\textwidth}
\begin{tikzpicture}
\draw(-0.5, 0.5) circle (0.5cm);
\node at (-0.5, 0.5) {5};
\draw(0.5, 0.5) circle (0.5cm);
\node at (0.5, 0.5) {6};
\draw(0.5, -0.5) circle (0.5cm);
\node at (0.5, -0.5){7};
\draw(-0.5, -0.5) circle (0.5cm);
\node at (-0.5, -0.5) {8};
\node[below] at (0,-1.5) {$b_5+b_7 = b_6 + b_8$};
\end{tikzpicture}
\end{subfigure}
\caption{Quadratic form and linear relations}
\label{fig:square2}
\end{figure}
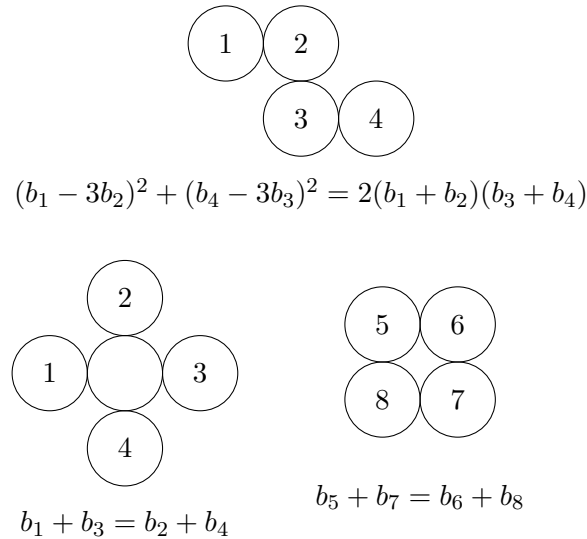

\begin{figure}[H]
\centering
\begin{subfigure}{.44\textwidth}
\begin{tikzpicture}[scale=1.3]
\draw (1,0) circle (1cm);

\draw[thick, blue] (0,1) circle (1cm);

\draw (-.5, 0) -- (1.5, 0);

\draw (0, -.5) -- (0, 1.5);

\draw (-.2, 1.2) -- (1.2, -.2);

\end{tikzpicture}
\caption{The orbit of the blue circle under the group generated by reflections across all the circles is the square superpacking.}
\end{subfigure}
\hspace{.1\textwidth}
\begin{subfigure}{.44\textwidth}
\begin{tikzpicture}[scale=1.3]
\draw (1,0) circle (1cm);

\draw[thick, blue] (0,1) circle (1cm);

\draw[thick, blue] (-.5, 0) -- (1.5, 0);

\draw (0, -.5) -- (0, 1.5);

\draw (-.5, 1) -- (1.5, 1);

\draw (1, -.5) -- (1, 1.5);

\end{tikzpicture}
\caption{The orbit of the blue circles under the group generated by reflections across all the circles is the Apollonian superpacking.}
\end{subfigure}
\caption{Commensurability of superpacking groups}
\label{fig:square3}
\end{figure}
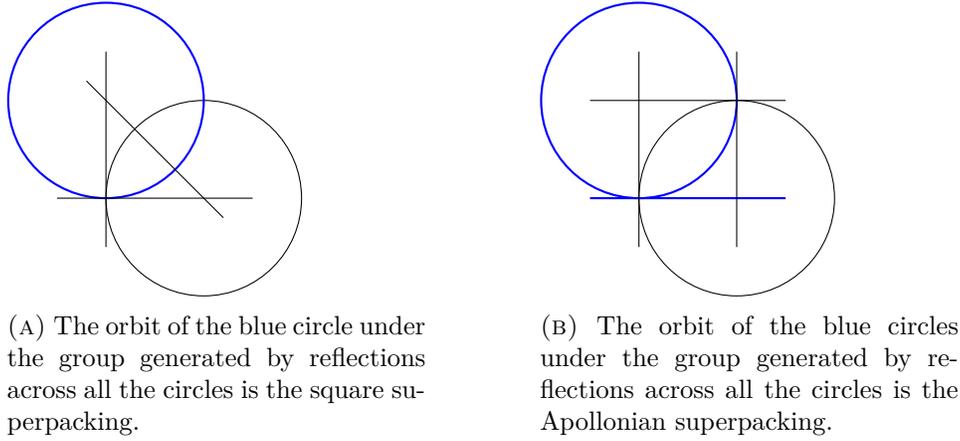

\subsection{Triangular and Hexagonal Packings}

The base configuration and dual configurations for the triangular packing are shown in Figure \ref{fig:triangular1}. In the hexagonal packing, the roles of base and dual configurations are reversed. Circles in the base and dual configuration for the triangular packing can be represented with coordinates in the sets 
\begin{align*}
&\left\lbrace (\tilde{b}, b, h_1, h_2)^T \in \R^4 \, \left| \, b, \tilde{b} \in \Z, \, h_1+h_2i \in 2 \Z\left[ \tfrac{1+i\sqrt{3}}{2} \right] \right. \right\rbrace, \\
&\left\lbrace (\tilde{b}, b, h_1, h_2)^T \in \R^4 \, \left| \, b, \tilde{b} \in \sqrt{3}\Z, \, h_1+h_2i \in 2i \Z\left[ \tfrac{1+i\sqrt{3}}{2} \right] \setminus 2\sqrt{3}\Z\left[ \tfrac{1+i\sqrt{3}}{2} \right] \right. \right\rbrace
\end{align*}
respectively. Reflection across the base and dual circles preserve these sets, so the triangular packing is superintegral, and rescaling by $\sqrt{3}$, we see that the hexagonal packing is superintegral as well. 

\begin{figure}[H]
\begin{center}
\fbox{
\begin{tikzpicture}[scale=1.4]
\clip (-2.5, {-sqrt(3)/2-.5}) -- (1.5, {-sqrt(3)/2-.5})-- (1.5, {sqrt(3)/2+.5}) -- (-2.5, {sqrt(3)/2+.5});

\draw[thick, blue](-2, 0) circle (0.5cm);
\draw[thick, blue](-1, 0) circle (0.5cm);
\draw[thick, blue](0, 0) circle (0.5cm);
\draw[thick, blue](1, 0) circle (0.5cm);
\draw[thick, blue](-5/2, {sqrt(3)/2}) circle (0.5cm);
\draw[thick, blue](-3/2, {sqrt(3)/2}) circle (0.5cm);
\draw[thick, blue](-1/2, {sqrt(3)/2}) circle (0.5cm);
\draw[thick, blue](1/2, {sqrt(3)/2}) circle (0.5cm);
\draw[thick, blue](3/2, {sqrt(3)/2}) circle (0.5cm);
\draw[thick, blue](-5/2, {-sqrt(3)/2}) circle (0.5cm);
\draw[thick, blue](-3/2, {-sqrt(3)/2}) circle (0.5cm);
\draw[thick, blue](-1/2, {-sqrt(3)/2}) circle (0.5cm);
\draw[thick, blue](1/2, {-sqrt(3)/2}) circle (0.5cm);
\draw[thick, blue](3/2, {-sqrt(3)/2}) circle (0.5cm);

\draw[thick, red](-2, {sqrt(3)/3}) circle ({sqrt(3)/6});
\draw[thick, red](-1, {sqrt(3)/3}) circle ({sqrt(3)/6});
\draw[thick, red](0, {sqrt(3)/3}) circle ({sqrt(3)/6});
\draw[thick, red](1, {sqrt(3)/3}) circle ({sqrt(3)/6});

\draw[thick, red](-2.5, {sqrt(3)/6}) circle ({sqrt(3)/6});
\draw[thick, red](-1.5, {sqrt(3)/6}) circle ({sqrt(3)/6});
\draw[thick, red](-.5, {sqrt(3)/6}) circle ({sqrt(3)/6});
\draw[thick, red](.5, {sqrt(3)/6}) circle ({sqrt(3)/6});
\draw[thick, red](1.5, {sqrt(3)/6}) circle ({sqrt(3)/6});

\draw[thick, red](-2.5, {-sqrt(3)/6}) circle ({sqrt(3)/6});
\draw[thick, red](-1.5, {-sqrt(3)/6}) circle ({sqrt(3)/6});
\draw[thick, red](-.5, {-sqrt(3)/6}) circle ({sqrt(3)/6});
\draw[thick, red](.5, {-sqrt(3)/6}) circle ({sqrt(3)/6});
\draw[thick, red](1.5, {-sqrt(3)/6}) circle ({sqrt(3)/6});

\draw[thick, red](-2, {-sqrt(3)/3}) circle ({sqrt(3)/6});
\draw[thick, red](-1, {-sqrt(3)/3}) circle ({sqrt(3)/6});
\draw[thick, red](0, {-sqrt(3)/3}) circle ({sqrt(3)/6});
\draw[thick, red](1, {-sqrt(3)/3}) circle ({sqrt(3)/6});

\draw[thick, red](-2, {2*sqrt(3)/3}) circle ({sqrt(3)/6});
\draw[thick, red](-1, {2*sqrt(3)/3}) circle ({sqrt(3)/6});
\draw[thick, red](0, {2*sqrt(3)/3}) circle ({sqrt(3)/6});
\draw[thick, red](1, {2*sqrt(3)/3}) circle ({sqrt(3)/6});

\draw[thick, red](-2, {-2*sqrt(3)/3}) circle ({sqrt(3)/6});
\draw[thick, red](-1, {-2*sqrt(3)/3}) circle ({sqrt(3)/6});
\draw[thick, red](0, {-2*sqrt(3)/3}) circle ({sqrt(3)/6});
\draw[thick, red](1, {-2*sqrt(3)/3}) circle ({sqrt(3)/6});

\draw[thick, red](-2.5, {5*sqrt(3)/6}) circle ({sqrt(3)/6});
\draw[thick, red](-1.5, {5*sqrt(3)/6}) circle ({sqrt(3)/6});
\draw[thick, red](-.5, {5*sqrt(3)/6}) circle ({sqrt(3)/6});
\draw[thick, red](.5, {5*sqrt(3)/6}) circle ({sqrt(3)/6});
\draw[thick, red](1.5, {5*sqrt(3)/6}) circle ({sqrt(3)/6});

\draw[thick, red](-2.5, {-5*sqrt(3)/6}) circle ({sqrt(3)/6});
\draw[thick, red](-1.5, {-5*sqrt(3)/6}) circle ({sqrt(3)/6});
\draw[thick, red](-.5, {-5*sqrt(3)/6}) circle ({sqrt(3)/6});
\draw[thick, red](.5, {-5*sqrt(3)/6}) circle ({sqrt(3)/6});
\draw[thick, red](1.5, {-5*sqrt(3)/6}) circle ({sqrt(3)/6});
\end{tikzpicture}}
\end{center}
\caption{Base and dual configurations for the triangular and hexagonal packings}
\label{fig:triangular1}
\end{figure}
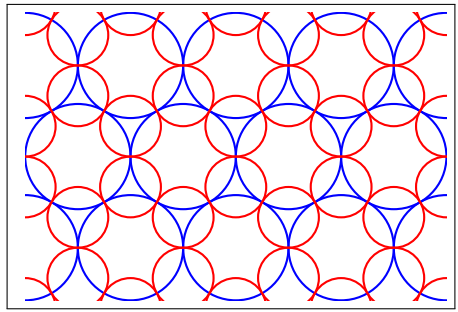

Quadratic and linear relations satisfied by curvatures in the triangular and hexagonal packings are shown in Figure \ref{fig:triangular2}. The symmetry group of the triangular and hexagonal superpackings is commensurate to the symmetry group of the limit of trapezohedral superpackings from \cite{Ahmed2}, as illustrated in Figure \ref{fig:triangular3}. We conjecture that this group is not commensurate to the symmetry group of any polyhedral superpacking. The full triangular packing is shown in Figure \ref{fig:triangular4} and the full hexagonal packing is shown in Figure \ref{fig:triangular5}. Note: the circles in  Figures \ref{fig:square} (B), \ref{fig:triangular4}, and \ref{fig:triangular5} were obtained by reflecting a subset of base circles across four generations of dual circles.

\begin{figure}[H]
\begin{center}
\begin{subfigure}{.6\textwidth}
\begin{center}
\begin{tikzpicture}[scale = 0.5]
\draw[label = "C1"] (0,0) circle (1cm);
\node at (0, 0) {1};
\draw[label = "C2"] (2,0) circle (1cm);
\node at (2, 0) {2};
\draw[label = "C3"] (1, {-sqrt(3)}) circle [radius = 1];
\node at (1, {-sqrt(3)}) {3};
\draw[label = "C4"] (-1, {-sqrt(3)}) circle [radius = 1];
\node at (-1, {-sqrt(3)}) {4};
\node at (0, -4) {$(3b_1-b_2+3b_3-b_4)^2 = 12b_1b_3+4b_2b_4$};
\end{tikzpicture}
\end{center}
\end{subfigure}

\vspace{.2in}

\begin{subfigure}{.59\textwidth}
\begin{center}
\begin{tikzpicture}[scale = 0.5]
\draw (0, 0) circle (1cm);
\node at (-2, 0){1};
\node at (1, {sqrt(3)}){2};
\node at (1, {-sqrt(3)}){3};
\node at (0, 0){4};
\draw (-2, 0) circle (1cm);
\draw (1, {sqrt(3)}) circle (1cm);
\draw (1, {-sqrt(3)}) circle (1cm);
\node at (0, -4) {$2b_1^2+2b_2^2+2b_3^2+10b_4^2=(b_1+b_2+b_3+b_4)^2$};
\end{tikzpicture}
\end{center}
\end{subfigure}
\hspace{.1\textwidth}
\begin{subfigure}{.29\textwidth}    
\begin{tikzpicture}[scale = 0.5]
\path (1, {sqrt(3)}) circle (1cm);
\draw[label = "C1"] (0,0) circle (1cm);
\node at (0, 0) {1};
\draw[label = "C2"] (2,0) circle (1cm);
\node at (2, 0) {2};
\draw(1, {-sqrt(3)}) circle (1cm);
\draw[label = "C3"] (3, {-sqrt(3)}) circle [radius = 1];
\node at (3, {-sqrt(3)}) {3};
\draw[label = "C4"] (-1, {-sqrt(3)}) circle [radius = 1];
\node at (-1, {-sqrt(3)}) {4};
\node at (1, -4) {$2b_1 - 2b_2 = b_4 - b_3$};
\end{tikzpicture}
\end{subfigure}
\end{center}
\caption{Quadratic forms and linear relation}
\label{fig:triangular2}
\end{figure}

\begin{figure}[H]
\centering
\begin{subfigure}{.44\textwidth}
\begin{tikzpicture}[scale=.9]
\path (0, 3) circle ({sqrt(3)});
\draw (-.5,0) -- ({sqrt(3)+.5},0);
\draw ({0-sqrt(3)/4},{1+1/4}) -- ({sqrt(3)+sqrt(3)/4},{0-1/4});
\draw (0,-.5) -- (0,1.5);
\draw[thick, blue] ({sqrt(3)}, 0) circle ({sqrt(3)});
\draw[thick, red] (0, 1) circle (1);
\end{tikzpicture}
\caption{The orbit of the blue circle is the triangular superpacking. The orbit of the red circle is the hexagonal superpacking.}
\end{subfigure}
\hspace{.1\textwidth}
\begin{subfigure}{.44\textwidth}
\begin{tikzpicture}[scale=.9]
\draw[thick, red] (-.5,0) -- ({sqrt(3)+.5},0);
\draw (0,-.5) -- (0,3.5);
\draw ({sqrt(3)},-.5) -- ({sqrt(3)},3.5);
\draw[thick, red] (-.5,3) -- ({sqrt(3)+.5},3);
\draw({sqrt(3)}, 0) circle ({sqrt(3)});
\draw(0,3) circle ({sqrt(3)});
\draw[thick, red] (0, 1) circle (1);
\draw[thick, red] ({sqrt(3)}, 2) circle (1);
\end{tikzpicture}
\caption{The orbit of the red circles is the limit of trapezohedral superpackings.\\}
\end{subfigure}
\caption{Commensurability of superpacking groups}
\label{fig:triangular3}
\end{figure}

\begin{figure}[H]
\begin{center}
\fbox{\includegraphics[width=.67\textwidth]{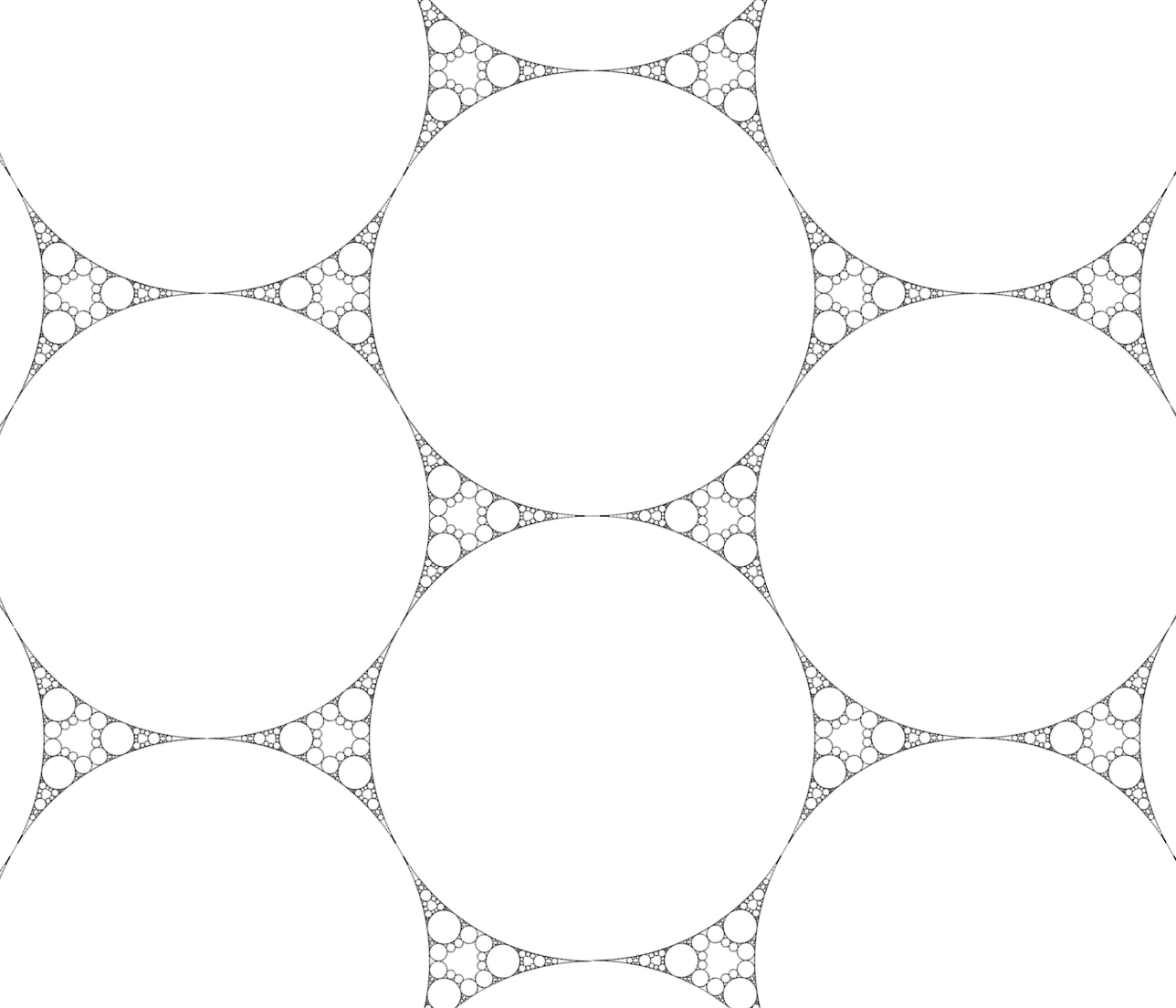}}
\end{center}
\caption{Triangular packing}
\label{fig:triangular4}
\end{figure}

\begin{figure}[H]\captionsetup{justification=centering,margin=2cm}

\begin{center}
\fbox{\includegraphics[width=.67\textwidth]{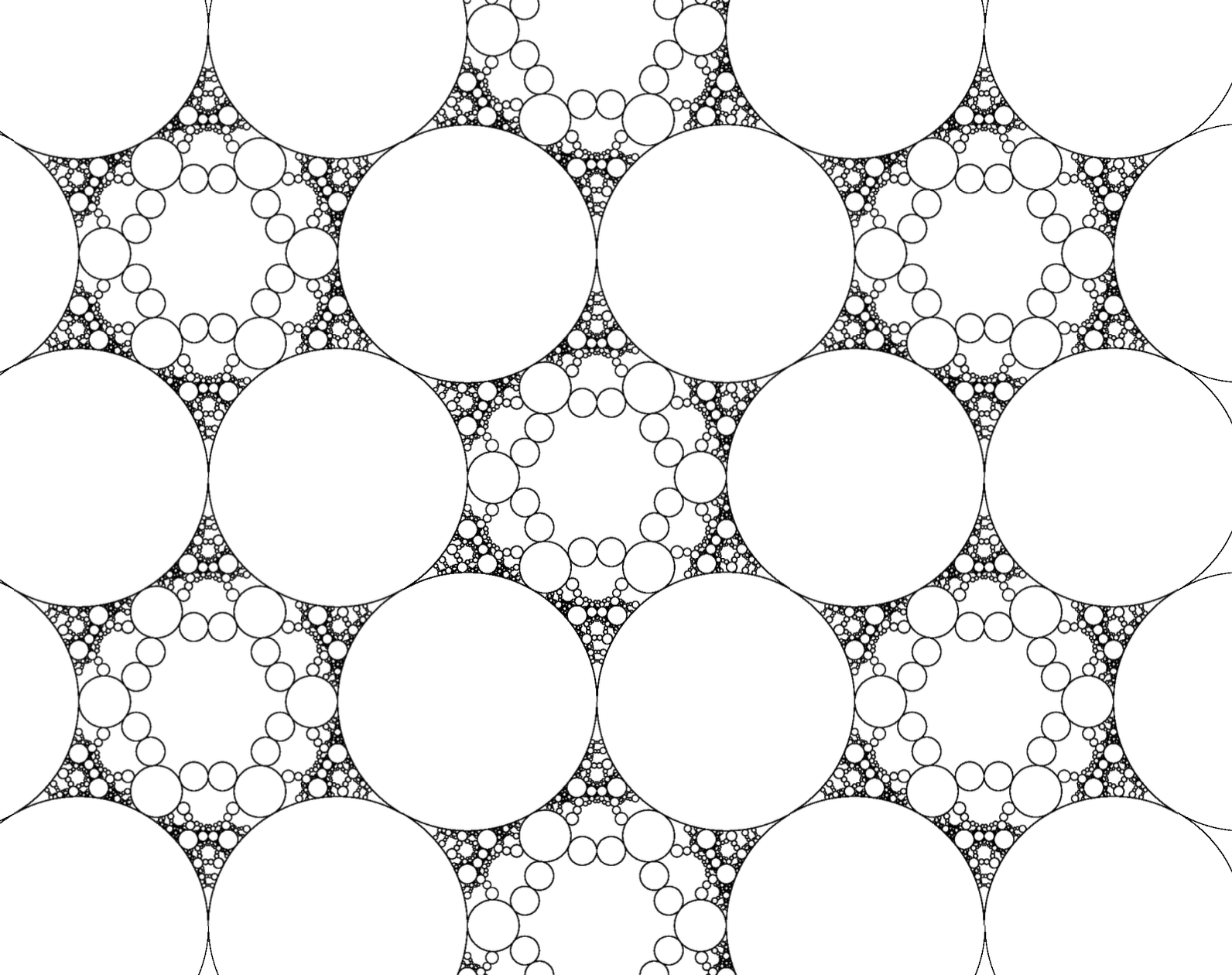}}
\end{center}
\caption{Hexagonal packing}

\label{fig:triangular5}
\end{figure}

We finish this section with some brief remarks on the limits of polyhedral packing families studied in \cite{Ahmed2}. The limit of trapezohedral packings is superintegral. As discussed above, the symmetry group of its superpacking is commensurate to that of the triangular and hexagonal superpackings. The limit of the cupola packings is also superintegral. The symmetry group of its superpacking is commensurate to that of the octahedral superpacking. The limit of the anticupola packings is not integral; this can be proven by an infinite descent argument. 

\section{Wallpaper Groups}

The goal of this section is to illustrate the rich variety of circle packings satisfying Definition \ref{packing}, with a focus on symmetry groups. We give the following converse to Theorem \ref{gamma2}:
\begin{theorem}
All the possible groups listed in Theorem \ref{gamma2} actually arise as the group of symmetries $\Gamma_2=\textsc{Sym}(B,\hat{B})$ for some base and dual configurations $B$, $\hat{B}$ satisfying Definition \ref{basedual}. 
\end{theorem}
\noindent In the case of finite $B$, $\hat{B}$, this theorem follows directly from the Koebe-Andreev-Thurston theorem. For infinite $B$, $\hat{B}$, the wallpaper group case is Theorem \ref{wallpaper} below. The other cases are simpler, and the proofs are omitted. 

\begin{theorem} \label{wallpaper}
Any wallpaper group is the symmetry group $\Gamma_2$ of the base and dual configurations $B$, $\hat{B}$ of some circle packing. Moreover, such $B$ can be realized as the refinement of the base configuration of either the triangular or square packing.
\end{theorem}

\begin{proof}
We will prove this theorem by illustration. We label the centers of rotations and axes of reflections using notation introduced in Figure \ref{wallpaper_label}. The entire base configurations can be generated by these symmetries or translations, starting from the circles shown.

\begin{figure}[H]
    \centering
    \begin{tikzpicture}
    \draw[red,thick] (-1,0) -- (0.5,0);
    \node[right] at (0.5,0) {Axis of reflection};
    \draw[blue,thick,dashed] (-1,-0.5) -- (0.5,-0.5);
    \node[right] at (0.5,-0.5) {Axis of glide reflection};
    \filldraw[gray] (-0.75,-1) circle (0.15);
    \node[right] at (-0.5,-1) {Center of rotation by $\pi$};
    \filldraw[violet] (-0.75,-1.3) -- (-0.55,-1.65) -- (-0.95,-1.65);
    \node[right] at (-0.5,-1.5) {Center of rotation by $2\pi/3$};
    \filldraw[cyan] (-0.9,-1.85) rectangle (-0.6,-2.15);
    \node[right] at (-0.5,-2) {Center of rotation by $\pi/2$};
    \filldraw[brown] (-0.95,-2.5) -- (-0.85,-2.33) -- (-0.65,-2.33) -- (-0.55, -2.5) -- (-0.65,-2.67) -- (-0.85,-2.67);
    \node[right] at (-0.5,-2.5) {Center of rotation by $\pi/3$};
    \end{tikzpicture}
    \caption{Labels for axes of reflection, centers of rotation}
    \label{wallpaper_label}
\end{figure}
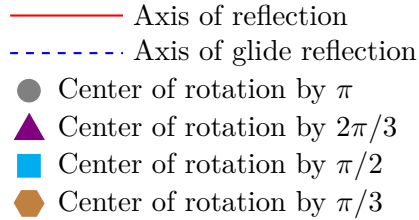

The wallpaper groups $p6m$ and $p4m$ are the most complicated ones, and they are the symmetry groups of the base configurations of triangular and square packings, respectively, as shown in Figure \ref{fig:maxwallpaper}. 

Refining these two base configurations by adding smaller circles removes symmetries. Thus, we can obtain base configurations with different symmetry groups by refining the triangular or square configurations. Notice that the refinements given below are well-defined base configurations because the associated dual configurations still exist. Refinements of the triangular configuration are shown in Figure \ref{fig:triangularrefinement} and refinements of the square configuration are shown in Figure \ref{fig:squarerefinement}. Some wallpaper groups could be obtained by refining either configuration, but we only show one realization of each group.

\begin{figure}[H]
\begin{center}
\begin{tabular}{c c}
\underline{Wallpaper Group $p6m$}
&  
\underline{Wallpaper Group $p4m$}
\vspace{2mm}
\\

\begin{tikzpicture}[scale=0.8]
\draw(0, 0) circle (1);
\draw(2, 0) circle (1);
\draw(1, {sqrt(3)}) circle (1);
\draw(3, {sqrt(3)}) circle (1);

\draw[red,thick] (-0.5,0) -- (2.5,0);
\draw[red,thick] (-0.25,-0.43) -- (1.25, {sqrt(3)+0.43});
\draw[red,thick] (2.25,-0.43) -- (0.75, {sqrt(3)+0.43});
\draw[red,thick] (1.75,-0.43) -- (3.25, {sqrt(3)+0.43});
\draw[red,thick] (0.5, {sqrt(3)}) -- (3.5, {sqrt(3)});
\draw[red,thick] (-0.43,-0.25) -- (3.43,{sqrt(3)+0.25});
\draw[red,thick] (2.43,-0.25) -- (0.07,{sqrt(3)/2+0.25});
\draw[red,thick] (1, {sqrt(3)+0.5}) -- (1,-0.5);
\draw[red,thick] (2, {sqrt(3)+0.5}) -- (2,-0.5);
\draw[red,thick] (0.57, {sqrt(3)+0.25}) -- (2.93,{sqrt(3)/2-0.25});
\draw[blue,thick,dashed] (0,{sqrt(3)/2}) -- (3,{sqrt(3)/2});
\draw[blue,thick,dashed] (0.25,{sqrt(3)/2+0.43}) -- (1.25,-0.43);
\draw[blue,thick,dashed] (0.75,-0.43) -- (2.25,{sqrt(3)+0.43});
\draw[blue,thick,dashed] (1.75,{sqrt(3)+0.43}) -- (2.75,{sqrt(3)/2-0.43});

\filldraw[gray] (1,0) circle (0.1);
\filldraw[gray] (2, {sqrt(3)}) circle (0.1);
\filldraw[gray] (1/2,{sqrt(3)/2}) circle (0.1);
\filldraw[gray] (3/2,{sqrt(3)/2}) circle (0.1);
\filldraw[gray] (5/2,{sqrt(3)/2}) circle (0.1);

\filldraw[brown] (-0.1, 0.17) -- (0.1,0.17) -- (0.2,0) -- (0.1,-0.17) -- (-0.1,-0.17) --(-0.2,0);
\filldraw[brown] (1.9, 0.17) -- (2.1,0.17) -- (2.2,0) -- (2.1,-0.17) -- (1.9,-0.17) --(1.8,0);
\filldraw[brown] (0.9, {sqrt(3)+0.17}) -- (1.1,{sqrt(3)+0.17}) -- (1.2,{sqrt(3)}) -- (1.1,{sqrt(3)-0.17}) -- (0.9,{sqrt(3)-0.17}) --(0.8,{sqrt(3)});
\filldraw[brown] (2.9, {sqrt(3)+0.17}) -- (3.1,{sqrt(3)+0.17}) -- (3.2,{sqrt(3)}) -- (3.1,{sqrt(3)-0.17}) -- (2.9,{sqrt(3)-0.17}) --(2.8,{sqrt(3)});
\filldraw[violet] (0.8,0.46) -- (1.2,0.46) -- (1,0.81);
\filldraw[violet] (1.8,1.03) -- (2.2,1.03) -- (2,1.38);
\end{tikzpicture}
& 
\begin{tikzpicture}[scale = 1.6]

\draw(0, 0) circle (0.5cm);
\draw(1, 0) circle (0.5cm);
\draw(0, 1) circle (0.5cm);
\draw(1, 1) circle (0.5cm);

\draw[red, thick] (-0.25, 0) -- (1.25, 0);
\draw[red, thick] (0, -0.25) -- (0, 1.25);
\draw[red, thick] (-0.25, 1) -- (1.25, 1);
\draw[red, thick] (1, -0.25) -- (1, 1.25);
\draw[red, thick] (-0.18, 1.18) -- (1.18, -0.18);
\draw[red, thick] (-0.18, -0.18) -- (1.18, 1.18); 
\draw[red, thick] (0.5,-0.25) -- (0.5,1.25);
\draw[red, thick] (-0.25,0.5) -- (1.25,0.5);
\draw[blue,thick,dashed] (-0.18,0.32) -- (0.68,1.18);
\draw[blue,thick,dashed] (1.18,0.32) -- (0.32,1.18);
\draw[blue,thick,dashed] (1.18,0.68) -- (0.32,-0.18);
\draw[blue,thick,dashed] (-0.18,0.68) -- (0.68,-0.18);

\filldraw[cyan] (-0.075, -0.075) rectangle (0.075, 0.075);
\filldraw[cyan] (-0.075, 0.925) rectangle (0.075, 1.075);
\filldraw[cyan] (0.925, 0.925) rectangle (1.075, 1.075);
\filldraw[cyan] (0.925, -0.075) rectangle (1.075, 0.075);
\filldraw[cyan] (0.575, 0.575) rectangle (0.425, 0.425);

\filldraw[gray] (0.5,1) circle (0.05);
\filldraw[gray] (0.5,0) circle (0.05);
\filldraw[gray] (0,0.5) circle (0.05);
\filldraw[gray] (1,0.5) circle (0.05);
\end{tikzpicture}
\end{tabular}
\end{center}
\caption{Maximal wallpaper groups}
\label{fig:maxwallpaper}
\end{figure}
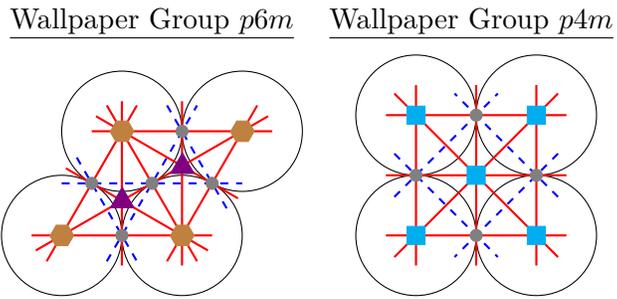

\begin{figure}[H]
\begin{center}
\begin{tabular}{c c}
\underline{Wallpaper Group $p3$}
    &
\underline{Wallpaper Group $p31m$}
\vspace{2mm}\\
\begin{tikzpicture}[scale=0.7]
\draw(-2, 0) circle (1cm);
\draw(0, 0) circle (1cm);
\draw(2, 0) circle (1cm);
\draw(-1, {sqrt(3)}) circle (1cm);
\draw(1, {sqrt(3)}) circle (1cm);
\draw(-1, {-sqrt(3)}) circle (1cm);
\draw(1, {-sqrt(3)}) circle (1cm);

\draw(0, {2*sqrt(3)/3}) circle ({2*sqrt(3)/3-1});
\draw(-1, -{sqrt(3)/3}) circle ({2*sqrt(3)/3-1});
\draw(1, -{sqrt(3)/3}) circle ({2*sqrt(3)/3-1});
\draw({(4*sqrt(3)-9)/11}, {2*sqrt(3)/3 - (3*sqrt(3)-4)/11}) circle ({1/(9+4*sqrt(3)});
\draw({-1+(9-4*sqrt(3))/11}, {-sqrt(3)/3 - (3*sqrt(3) - 4)/11)}) circle ({1/(9+4*sqrt(3)});
\draw(1, {-1+sqrt(3)/3+1/(9+4*sqrt(3))}) circle ({1/(9+4*sqrt(3))});

\filldraw[violet] (1.8,-0.12) -- (2.2,-0.12) -- (2,0.23);
\filldraw[violet] (-1.2,{sqrt(3)-0.12}) -- (-0.8,{sqrt(3)-0.12}) -- (-1,{sqrt(3)+0.23});
\filldraw[violet] (-1.2,{-sqrt(3)-0.12}) -- (-0.8,{-sqrt(3)-0.12}) -- (-1,{-sqrt(3)+0.23});
\filldraw[violet] (-0.2,-0.12) -- (0.2,-0.12) -- (0,0.23);
\filldraw[violet] (0.8,{sqrt(3)-0.12}) -- (1.2,{sqrt(3)-0.12}) -- (1,{sqrt(3)+0.23});
\filldraw[violet] (0.8,{-sqrt(3)-0.12}) -- (1.2,{-sqrt(3)-0.12}) -- (1,{-sqrt(3)+0.23});
\filldraw[violet] (-2.2,-0.12) -- (-1.8,-0.12) -- (-2,0.23);

\end{tikzpicture}
    &
\begin{tikzpicture}[scale=0.7]
\draw(-2, 0) circle (1cm);
\draw(0, 0) circle (1cm);
\draw(2, 0) circle (1cm);
\draw(-1, {sqrt(3)}) circle (1cm);
\draw(1, {sqrt(3)}) circle (1cm);
\draw(-1, {-sqrt(3)}) circle (1cm);
\draw(1, {-sqrt(3)}) circle (1cm);

\draw(0, {2*sqrt(3)/3}) circle ({2*sqrt(3)/3-1});
\draw(-1, -{sqrt(3)/3}) circle ({2*sqrt(3)/3-1});
\draw(1, -{sqrt(3)/3}) circle ({2*sqrt(3)/3-1});
\draw({(4*sqrt(3)-9)/11}, {2*sqrt(3)/3 - (3*sqrt(3)-4)/11}) circle ({1/(9+4*sqrt(3)});
\draw({-1+(9-4*sqrt(3))/11}, {-sqrt(3)/3 - (3*sqrt(3) - 4)/11)}) circle ({1/(9+4*sqrt(3)});
\draw(1, {-1+sqrt(3)/3+1/(9+4*sqrt(3))}) circle ({1/(9+4*sqrt(3))});

\draw[red,thick] (-2.75,-1.3) -- (-0.25,3.03);
\draw[red,thick] (2.75,-1.3) -- (0.25,3.03);
\draw[red,thick] (-2.5,{-sqrt(3)}) -- (2.5,{-sqrt(3)});
\draw[blue,thick,dashed] (-2.7,{sqrt(3)/2}) -- (2.7,{sqrt(3)/2});
\draw[blue,thick,dashed] (-2.15,1.99) -- (0.65,-2.86);
\draw[blue,thick,dashed] (2.15,1.99) -- (-0.65,-2.86);
\filldraw[violet] (-0.2,-0.12) -- (0.2,-0.12) -- (0,0.23);
\end{tikzpicture}
\end{tabular}
\end{center}

\begin{center}
\begin{tabular}{c c}
\underline{Wallpaper Group $p3m1$}
    &
\underline{Wallpaper Group $p6$}
\vspace{2mm}\\
\begin{tikzpicture}[scale=0.7]
\draw(-2, 0) circle (1cm);
\draw(0, 0) circle (1cm);
\draw(2, 0) circle (1cm);

\draw(-1, {sqrt(3)}) circle (1cm);
\draw(1, {sqrt(3)}) circle (1cm);
\draw(-1, {-sqrt(3)}) circle (1cm);
\draw(1, {-sqrt(3)}) circle (1cm);

\draw[red,thick] (0,-2.5) -- (0,2.5);
\draw[red,thick] (-2.25,-1.3) -- (2.25,1.3);
\draw[red,thick] (2.25,-1.3) -- (-2.25,1.3);
\draw[red,thick] (-2,1.4) -- (-2,-1.4);
\draw[red,thick] (2,1.4) -- (2,-1.4);
\draw[red,thick] (-2.23,1.022) -- (0.23,2.44);
\draw[red,thick] (2.23,1.022) -- (-0.23,2.44);
\draw[red,thick] (-2.23,-1.022) -- (0.23,-2.44);
\draw[red,thick] (2.23,-1.022) -- (-0.23,-2.44);
\draw[blue,thick,dashed] (1,{sqrt(3)+.5}) -- (1,{-sqrt(3)-.5});
\draw[blue,thick,dashed] (-1,{sqrt(3)+.5}) -- (-1,{-sqrt(3)-.5});
\draw[blue,thick,dashed] (-2.45,-.3) -- (1.45,{sqrt(3)+.3});
\draw[blue,thick,dashed] (2.45,-.3) -- (-1.45,{sqrt(3)+.3});
\draw[blue,thick,dashed] (-2.45,0.3) -- (1.45,{-sqrt(3)-.3});
\draw[blue,thick,dashed] (2.45,.3) -- (-1.45,{-sqrt(3)-.3});
\filldraw[violet] (1.8,-0.12) -- (2.2,-0.12) -- (2,0.23);
\filldraw[violet] (-1.2,{sqrt(3)-0.12}) -- (-0.8,{sqrt(3)-0.12}) -- (-1,{sqrt(3)+0.23});
\filldraw[violet] (-1.2,{-sqrt(3)-0.12}) -- (-0.8,{-sqrt(3)-0.12}) -- (-1,{-sqrt(3)+0.23});
\filldraw[violet] (-0.2,-0.12) -- (0.2,-0.12) -- (0,0.23);
\filldraw[violet] (0.8,{sqrt(3)-0.12}) -- (1.2,{sqrt(3)-0.12}) -- (1,{sqrt(3)+0.23});
\filldraw[violet] (0.8,{-sqrt(3)-0.12}) -- (1.2,{-sqrt(3)-0.12}) -- (1,{-sqrt(3)+0.23});
\filldraw[violet] (-2.2,-0.12) -- (-1.8,-0.12) -- (-2,0.23);

\draw[fill=white](0, {2*sqrt(3)/3}) circle ({2*sqrt(3)/3-1});
\draw[fill=white](-1, -{sqrt(3)/3}) circle ({2*sqrt(3)/3-1});
\draw[fill=white](1, -{sqrt(3)/3}) circle ({2*sqrt(3)/3-1});
\end{tikzpicture}
    &
\begin{tikzpicture}[scale=0.7]
\draw(-2, 0) circle (1cm);
\draw(0, 0) circle (1cm);
\draw(2, 0) circle (1cm);
\draw(-1, {sqrt(3)}) circle (1cm);
\draw(1, {sqrt(3)}) circle (1cm);
\draw(-1, {-sqrt(3)}) circle (1cm);
\draw(1, {-sqrt(3)}) circle (1cm);

\draw(0, {2*sqrt(3)/3}) circle ({2*sqrt(3)/3-1});
\draw(-1, -{sqrt(3)/3}) circle ({2*sqrt(3)/3-1});
\draw(1, -{sqrt(3)/3}) circle ({2*sqrt(3)/3-1});
\draw(1, {sqrt(3)/3}) circle ({2*sqrt(3)/3-1});
\draw(0, {-2*sqrt(3)/3}) circle ({2*sqrt(3)/3-1});
\draw(-1, {sqrt(3)/3}) circle ({2*sqrt(3)/3-1});

\draw({(4*sqrt(3)-9)/11}, {2*sqrt(3)/3 - (3*sqrt(3)-4)/11}) circle ({1/(9+4*sqrt(3)});

\draw(-1, {1-sqrt(3)/3 - 1/(9+4*sqrt(3)}) circle ({1/(9+4*sqrt(3)});

\draw({-1+(9-4*sqrt(3))/11}, {-sqrt(3)/3 - (3*sqrt(3) - 4)/11)}) circle ({1/(9+4*sqrt(3)});

\draw({(9-4*sqrt(3))/11}, {-2*sqrt(3)/3 + (3*sqrt(3)-4)/11}) circle ({1/(9+4*sqrt(3)});

\draw(1, {-1+sqrt(3)/3+1/(9+4*sqrt(3))}) circle ({1/(9+4*sqrt(3))});

\draw({1-(9-4*sqrt(3))/11}, {sqrt(3)/3 + (3*sqrt(3) - 4)/11)}) circle ({1/(9+4*sqrt(3)});

\filldraw[brown] (-0.1, 0.17) -- (0.1,0.17) -- (0.2,0) -- (0.1,-0.17) -- (-0.1,-0.17) --(-0.2,0);
\filldraw[violet] (-1.2,{sqrt(3)-0.12}) -- (-0.8,{sqrt(3)-0.12}) -- (-1,{sqrt(3)+0.23});
\filldraw[violet] (-1.2,{-sqrt(3)-0.12}) -- (-0.8,{-sqrt(3)-0.12}) -- (-1,{-sqrt(3)+0.23});
\filldraw[violet] (0.8,{sqrt(3)-0.12}) -- (1.2,{sqrt(3)-0.12}) -- (1,{sqrt(3)+0.23});
\filldraw[violet] (0.8,{-sqrt(3)-0.12}) -- (1.2,{-sqrt(3)-0.12}) -- (1,{-sqrt(3)+0.23});
\filldraw[violet] (-2.2,-0.12) -- (-1.8,-0.12) -- (-2,0.23);
\filldraw[violet] (1.8,-0.12) -- (2.2,-0.12) -- (2,0.23);
\filldraw[gray] (0,{sqrt(3)}) circle (0.1);
\filldraw[gray] (3/2,{sqrt(3)/2}) circle (0.1);
\filldraw[gray] (3/2,{-sqrt(3)/2}) circle (0.1);
\filldraw[gray] (0,{-sqrt(3)}) circle (0.1);
\filldraw[gray] (-3/2,{-sqrt(3)/2}) circle (0.1);
\filldraw[gray] (-3/2,{sqrt(3)/2}) circle (0.1);
\end{tikzpicture}
\end{tabular}
\end{center}
\caption{Refinements of the triangular base configuration}
\label{fig:triangularrefinement}
\end{figure}
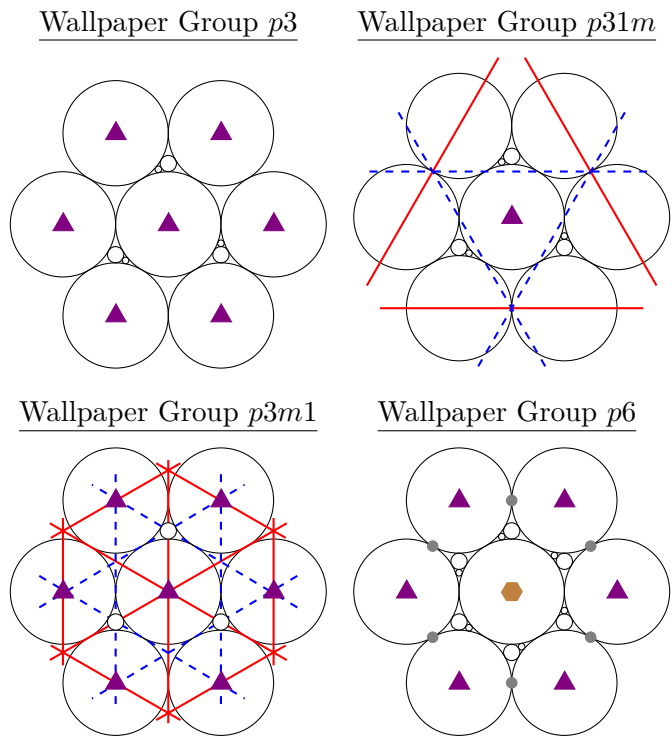

\begin{figure}
\begin{center}
\begin{tabular}{c c c}
\underline{Wallpaper Group $p1$}
    &
\underline{Wallpaper Group $p2$}
    &
\underline{Wallpaper Group $pm$}
\vspace{2mm}
\\

\begin{tikzpicture}[scale=0.6]
\draw(-2, -2) circle (1cm);
\draw(-2, 0) circle (1cm);
\draw(-2, 2) circle (1cm);
\draw(0, -2) circle (1cm);
\draw(0, 0) circle (1cm);
\draw(0, 2) circle (1cm);
\draw(2, -2) circle (1cm);
\draw(2, 0) circle (1cm);
\draw(2, 2) circle (1cm);

\draw(-1, 1) circle({sqrt(2)-1});
\draw(-1, {(9-4*sqrt(2))/7}) circle({(5-3*sqrt(2))/7});
\draw(1, 1) circle({sqrt(2)-1});
\draw(1, {(9-4*sqrt(2))/7}) circle({(5-3*sqrt(2))/7});
\draw(1, -1) circle({sqrt(2)-1});
\draw({(9-4*sqrt(2))/7}, -1) circle({(5-3*sqrt(2))/7});
\draw(-1, -1) circle({sqrt(2)-1});
\draw({(-4*sqrt(2)-5)/7}, -1) circle({(5-3*sqrt(2))/7});

\end{tikzpicture}
    & 
\begin{tikzpicture}[scale=0.6]
\draw(-2, -2) circle (1cm);
\draw(-2, 0) circle (1cm);
\draw(-2, 2) circle (1cm);
\draw(0, -2) circle (1cm);
\draw(0, 0) circle (1cm);
\draw(0, 2) circle (1cm);
\draw(2, -2) circle (1cm);
\draw(2, 0) circle (1cm);
\draw(2, 2) circle (1cm);

\draw(-1, 1) circle({sqrt(2)-1});
\draw(-1, {(9-4*sqrt(2))/7}) circle({(5-3*sqrt(2))/7});
\draw(1, -1) circle({sqrt(2)-1});
\draw(1, {-(9-4*sqrt(2))/7}) circle({(5-3*sqrt(2))/7});

\filldraw[gray] (-2,-2) circle (0.1);
\filldraw[gray] (-2,0) circle (0.1);
\filldraw[gray] (-2,2) circle (0.1);
\filldraw[gray] (0,-2) circle (0.1);
\filldraw[gray] (0,0) circle (0.1);
\filldraw[gray] (0,2) circle (0.1);
\filldraw[gray] (2,-2) circle (0.1);
\filldraw[gray] (2,0) circle (0.1);
\filldraw[gray] (2,2) circle (0.1);

\end{tikzpicture}
    &






\begin{tikzpicture}[scale = 0.6]
\draw(-2, 0) circle (1cm);
\draw(-2, 2) circle (1cm);
\draw(0, 0) circle (1cm);
\draw(0, 2) circle (1cm);
\draw(-2, -2) circle (1cm);
\draw(0, -2) circle (1cm);
\draw(2, -2) circle (1cm);
\draw(2, 0) circle (1cm);
\draw(2, 2) circle (1cm);
\draw(-1, 1) circle({sqrt(2)-1});
\draw(-1, {(4*sqrt(2)+5)/7}) circle({(5-3*sqrt(2))/7});
\draw[red,thick] (-2,-2.5) -- (-2,2.5);
\draw[red,thick] (2,-2.5) -- (2,2.5);
\end{tikzpicture}

\end{tabular}
\end{center}

\begin{center}
\begin{tabular}{c c c}
\underline{Wallpaper Group $pg$}
    &
\underline{Wallpaper Group $cm$}
    &  
\underline{Wallpaper Group $pmm$}
\vspace{2mm}
\\
\begin{tikzpicture}[scale=0.6]
\draw(-2, -2) circle (1cm);
\draw(-2, 0) circle (1cm);
\draw(-2, 2) circle (1cm);
\draw(0, -2) circle (1cm);
\draw(0, 0) circle (1cm);
\draw(0, 2) circle (1cm);
\draw(2, -2) circle (1cm);
\draw(2, 0) circle (1cm);
\draw(2, 2) circle (1cm);

\draw(-1, 1) circle({sqrt(2)-1});
\draw(1, -1) circle({sqrt(2)-1});
\draw({-(4*sqrt(2)+ 5)/7}, 1) circle({(5-3*sqrt(2))/7});
\draw(-1, {(4*sqrt(2)+5)/7}) circle({(5-3*sqrt(2))/7});
\draw({(4*sqrt(2)+ 5)/7}, -1) circle({(5-3*sqrt(2))/7});
\draw(1, {(4*sqrt(2)-9)/7}) circle({(5-3*sqrt(2))/7});

\draw[blue,thick,dashed] (-2,2.5) -- (-2,-2.5);
\draw[blue,thick,dashed] (0,2.5) -- (0,-2.5);
\draw[blue,thick,dashed] (2,2.5) -- (2,-2.5);

\end{tikzpicture}
    &
\begin{tikzpicture}[scale = 0.6]

\draw(-2, -2) circle (1cm);
\draw(-2, 0) circle (1cm);
\draw(-2, 2) circle (1cm);
\draw(0, -2) circle (1cm);
\draw(0, 0) circle (1cm);
\draw(0, 2) circle (1cm);
\draw(2, -2) circle (1cm);
\draw(2, 0) circle (1cm);
\draw(2, 2) circle (1cm);

\draw(1, 1) circle({sqrt(2)-1});
\draw(-1, -1) circle({sqrt(2)-1});

\draw({(-4*sqrt(2)- 5)/7}, -1) circle({(5-3*sqrt(2))/7});
\draw({(4*sqrt(2)+5)/7}, 1) circle({(5-3*sqrt(2))/7});
\draw(-1, {(4*sqrt(2)- 9)/7}) circle({(5-3*sqrt(2))/7});
\draw(1, {(4*sqrt(2)+5)/7}) circle({(5-3*sqrt(2))/7});

\draw[red,thick] (-2,2.5) -- (-2,-2.5);
\draw[blue,thick,dashed] (0,2.5) -- (0,-2.5);
\draw[red,thick] (2,2.5) -- (2,-2.5);



\end{tikzpicture}
    &





\begin{tikzpicture}[scale = 0.6]
\draw(-2, 0) circle (1cm);
\draw(-2, 2) circle (1cm);
\draw(0, 0) circle (1cm);
\draw(0, 2) circle (1cm);
\draw(-2, -2) circle (1cm);
\draw(0, -2) circle (1cm);
\draw(2, -2) circle (1cm);
\draw(2, 0) circle (1cm);
\draw(2, 2) circle (1cm);
\draw(-1, 1) circle({sqrt(2)-1});
\draw(-1, {(4*sqrt(2)+5)/7}) circle({(5-3*sqrt(2))/7});
\draw[red,thick] (-2.5,2) -- (2.5,2);
\draw[red,thick] (-2.5,-2) -- (2.5,-2);
\draw[red,thick] (-2,-2.5) -- (-2,2.5);
\draw[red,thick] (2,-2.5) -- (2,2.5);
\filldraw[gray] (2,2) circle (0.1);
\filldraw[gray] (2,-2) circle (0.1);
\filldraw[gray] (-2,2) circle (0.1);
\filldraw[gray] (-2,-2) circle (0.1);
\end{tikzpicture}
\end{tabular}
\end{center}

\begin{center}
\begin{tabular}{c c c}
\underline{Wallpaper Group $pmg$}
    &
\underline{Wallpaper Group $pgg$}
    &
\underline{Wallpaper Group $cmm$}
\vspace{2mm}\\

\begin{tikzpicture}[scale = 0.6]
\draw(-2, 0) circle (1cm);
\draw(-2, 2) circle (1cm);
\draw(0, 0) circle (1cm);
\draw(0, 2) circle (1cm);
\draw(-2, -2) circle (1cm);
\draw(0, -2) circle (1cm);
\draw(2, -2) circle (1cm);
\draw(2, 0) circle (1cm);
\draw(2, 2) circle (1cm);
\draw(-1, 1) circle({sqrt(2)-1});
\draw(-1, {(4*sqrt(2)+5)/7}) circle({(5-3*sqrt(2))/7});
\draw[red,thick] (-2.5,2) -- (2.5,2);
\draw[red,thick] (-2.5,-2) -- (2.5,-2);
\draw[blue,thick,dashed] (-2,-2.5) -- (-2,2.5);
\draw[blue,thick,dashed] (2,-2.5) -- (2,2.5);
\filldraw[gray] (2,0) circle (0.1);
\filldraw[gray] (-2,0) circle (0.1);
\end{tikzpicture}
&
\begin{tikzpicture}[scale = 0.6]
\draw(-2, 0) circle (1cm);
\draw(-2, 2) circle (1cm);
\draw(0, 0) circle (1cm);
\draw(0, 2) circle (1cm);
\draw(-2, -2) circle (1cm);
\draw(0, -2) circle (1cm);
\draw(2, -2) circle (1cm);
\draw(2, 0) circle (1cm);
\draw(2, 2) circle (1cm);

\draw(-1, 1) circle({sqrt(2)-1});
\draw(-1, {(4*sqrt(2)+5)/7}) circle({(5-3*sqrt(2))/7});



\draw[blue,thick,dashed] (-2.5,0) -- (2.5,0);
\draw[blue,thick,dashed] (0,2.5) -- (0,-2.5);
\filldraw[gray] (2,2) circle (0.1);
\filldraw[gray] (2,-2) circle (0.1);
\filldraw[gray] (-2,2) circle (0.1);
\filldraw[gray] (-2,-2) circle (0.1);
\end{tikzpicture}
    &
\begin{tikzpicture}[scale = 0.6]
\draw(-2, -2) circle (1cm);
\draw(-2, 0) circle (1cm);
\draw(-2, 2) circle (1cm);
\draw(0, -2) circle (1cm);
\draw(0, 0) circle (1cm);
\draw(0, 2) circle (1cm);
\draw(2, -2) circle (1cm);
\draw(2, 0) circle (1cm);
\draw(2, 2) circle (1cm);

\draw(-1, 1) circle({sqrt(2)-1});
\draw(-1, {(9-4*sqrt(2))/7}) circle({(5-3*sqrt(2))/7});
\draw(1, -1) circle({sqrt(2)-1});
\draw(1, {-(9-4*sqrt(2))/7}) circle({(5-3*sqrt(2))/7});

\draw[red,thick] (-2.5,2) -- (2.5,2);
\draw[red,thick] (-2.5,-2) -- (2.5,-2);
\draw[red,thick] (-2,2.5) -- (-2,-2.5);
\draw[red,thick] (2,-2.5) -- (2,2.5);
\draw[blue,thick,dashed] (0,2.5) -- (0,-2.5);
\draw[blue,thick,dashed] (2.5,0) -- (-2.5,0);
\filldraw[gray] (2,2) circle (0.1);
\filldraw[gray] (-2,2) circle (0.1);
\filldraw[gray] (2,-2) circle (0.1);
\filldraw[gray] (-2,-2) circle (0.1);
\filldraw[gray] (0,0) circle (0.1);
\end{tikzpicture}
\end{tabular}
\end{center}

\begin{center}
\begin{tabular}{c c}
\underline{Wallpaper Group $p4$}
    &
\underline{Wallpaper Group $p4g$}
\vspace{2mm}\\
\begin{tikzpicture}[scale=0.6]
\draw(-2, -2) circle (1cm);
\draw(-2, 0) circle (1cm);
\draw(-2, 2) circle (1cm);
\draw(0, -2) circle (1cm);
\draw(0, 0) circle (1cm);
\draw(0, 2) circle (1cm);
\draw(2, -2) circle (1cm);
\draw(2, 0) circle (1cm);
\draw(2, 2) circle (1cm);

\draw(-1, 1) circle({sqrt(2)-1});
\draw(-1, {(4*sqrt(2)+5)/7}) circle({(5-3*sqrt(2))/7});
\draw(1, 1) circle({sqrt(2)-1});
\draw({(4*sqrt(2)+5)/7}, 1) circle({(5-3*sqrt(2))/7});
\draw(-1, -1) circle({sqrt(2)-1});
\draw({(-5-4*sqrt(2))/7}, -1) circle({(5-3*sqrt(2))/7});
\draw(1, -1) circle({sqrt(2)-1});
\draw(1, {-(4*sqrt(2)+5)/7}) circle({(5-3*sqrt(2))/7});

\filldraw[cyan] (-0.15,-0.15) rectangle (0.15,0.15);
\filldraw[cyan] (-2.15,-2.15) rectangle (-1.85,-1.85);
\filldraw[cyan] (-2.15,1.85) rectangle (-1.85,2.15);
\filldraw[cyan] (1.85,-2.15) rectangle (2.15,-1.85);
\filldraw[cyan] (1.85,1.85) rectangle (2.15,2.15);
\filldraw[gray] (2,0) circle (0.1);
\filldraw[gray] (0,2) circle (0.1);
\filldraw[gray] (-2,0) circle (0.1);
\filldraw[gray] (0,-2) circle (0.1);
\end{tikzpicture}
    &

\begin{tikzpicture}[scale=0.6]
\draw(-2, -2) circle (1cm);
\draw(-2, 0) circle (1cm);
\draw(-2, 2) circle (1cm);
\draw(0, -2) circle (1cm);
\draw(0, 0) circle (1cm);
\draw(0, 2) circle (1cm);
\draw(2, -2) circle (1cm);
\draw(2, 0) circle (1cm);
\draw(2, 2) circle (1cm);

\draw(-1, 1) circle({sqrt(2)-1});
\draw(-1, {(4*sqrt(2)+5)/7}) circle({(5-3*sqrt(2))/7});
\draw(1, 1) circle({sqrt(2)-1});
\draw({(4*sqrt(2)+5)/7}, 1) circle({(5-3*sqrt(2))/7});
\draw(-1, -1) circle({sqrt(2)-1});
\draw({(-5-4*sqrt(2))/7}, -1) circle({(5-3*sqrt(2))/7});
\draw(1, -1) circle({sqrt(2)-1});
\draw(1, {-(4*sqrt(2)+5)/7}) circle({(5-3*sqrt(2))/7});

\draw[blue,thick,dashed] (0,2.5) -- (0,-2.5);
\draw[blue,thick,dashed] (2.5,0) -- (-2.5,0);
\draw[red,thick] (2.5,2) -- (-2.5,2);
\draw[red,thick] (2.5,-2) -- (-2.5,-2);
\draw[red,thick] (2,2.5) -- (2,-2.5);
\draw[red,thick] (-2,2.5) -- (-2,-2.5);

\filldraw[cyan] (-0.15,-0.15) rectangle (0.15,0.15);
\filldraw[gray] (2,2) circle (0.1);
\filldraw[gray] (-2,2) circle (0.1);
\filldraw[gray] (2,-2) circle (0.1);
\filldraw[gray] (-2,-2) circle (0.1);

\end{tikzpicture}
\end{tabular}
\end{center}
\caption{Refinements of the square base configuration}
\label{fig:squarerefinement}
\end{figure}
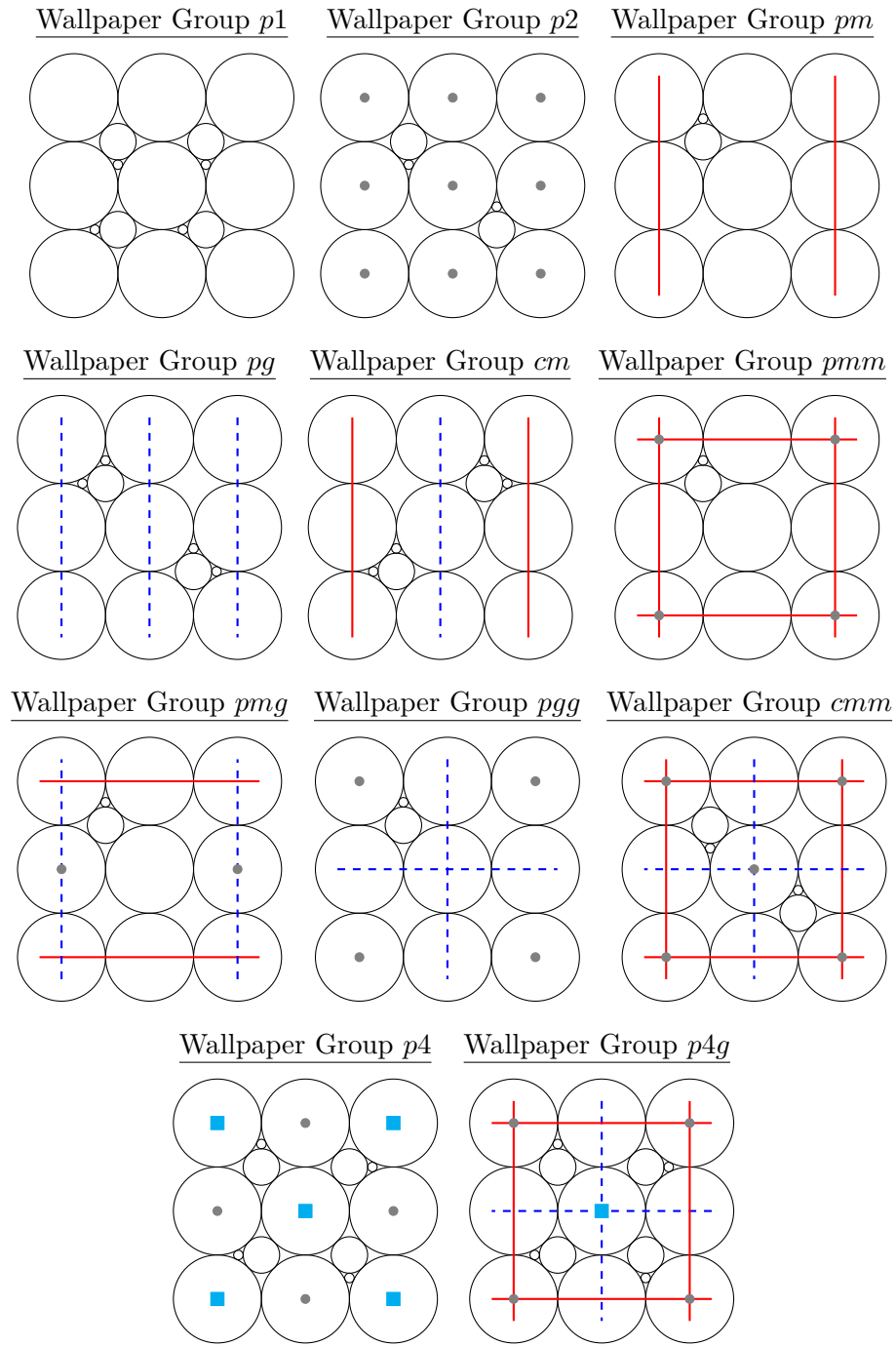
\end{proof}
\vspace{1in}

\bibliography{ApollonianBib}
\bibliographystyle{amsplain}

\end{document}